\documentclass[12pt]{amsart}
\usepackage{amssymb}
\usepackage{amsmath} 
\usepackage{amsthm} 
\usepackage{calc} 
\usepackage{verbatim} 
\usepackage[dvips]{graphics}
\usepackage[dvips]{epsfig}

\newtheorem{thm}{Theorem}[section]
\newtheorem{lem}[thm]{Lemma}

\newtheorem{cor}[thm]{Corollary}
\newtheorem{rem}[thm]{Remark}
\newtheorem{eg}[thm]{Example}
\newtheorem{prop}[thm]{Proposition}

\newcounter{probno}
\stepcounter{probno}

\newcommand{\self}{\circlearrowleft}

\newcommand{\eqdef}{:=}

\newcommand{\id}{\mathrm{id}}

\newcommand{\norm}[2][{}]{\left\|#2\right\|_{#1}}

\newcommand{\pair}[2]{\left\langle #1,#2 \right\rangle}

\newcommand{\R}{\mathbf{R}}

\newcommand{\C}{\mathbf{C}}
\newcommand{\cp}{\mathbf{P}}
\newcommand{\Z}{\mathbf{Z}}
\newcommand{\N}{\mathbf{N}}

\newcommand{\pic}{\mathrm{Pic}\,}

\newcommand{\Pic}{\operatorname{Pic}}
\newcommand{\pico}{\operatorname{Pic^0}}
\newcommand{\Aut}{\operatorname{Aut}}
\renewcommand{\eqdef}{:=}
\newtheorem{theo}{Theorem}

\newcommand{\creg}{C_{reg}}
\newcommand{\calO}{\mathcal{O}}
\newcommand{\calK}{\mathcal{K}}
\newcommand{\calL}{\mathcal{L}}
\newcommand{\bbK}{\mathbb{K}}

\newcommand{\frako}{\mathfrak{o}}
\newcommand{\frakm}{\mathfrak{m}}

\newcommand{\frakb}{\mathfrak{b}}
\newcommand{\Div}{\textup{Div}}
\newcommand{\ddiv}{\textup{div}}
\newcommand{\beq}{\begin{equation}}
\newcommand{\eeq}{\end{equation}}
\newcommand{\gr}{\textup{gr}}
\newcommand{\cha}{\textup{char}}

\newcommand{\la}{\langle}
\newcommand{\ra}{\rangle}

\title{Cremona transformations, surface automorphisms and plane cubics}

\date{\today}

\author{Jeffrey Diller}
\address{Department of Mathematics\\
         University of Notre Dame\\
         Notre Dame, IN 46556}
\email{diller.1@nd.edu}

\thanks{This work was supported in part by National Science Foundation grant DMS 06-53678.}
\begin{document}
\begin{abstract}
We give a method for constructing many examples of automorphisms with positive entropy on rational complex surfaces. The general idea is to begin with a quadratic Cremona transformation that fixes some cubic curve and then use the `group law' on the cubic to understand when the indeterminacy and exceptional behavior of the transformation can be eliminated by repeated blowing up.
\end{abstract}

\maketitle

\section*{Introduction}
Every automorphism of the complex projective plane $\cp^2$ is linear and therefore behaves quite simply when iterated.  It is natural to seek other rational complex surfaces, for instance those obtained from $\cp^2$ by successive blowing up, that admit automorphisms with more interesting dynamics.  Until recently, very few examples with positive entropy seem to have been known (see e.g. the introduction to \cite{can}).

Bedford and Kim \cite{bk} found some new examples by studying an explicit family of Cremona transformations, i.e. birational self-maps of $\cp^2$.  McMullen \cite{mcm} gave a more synthetic construction of some similar examples.  To this end he used the theory of infinite Coxeter groups, some results of Nagata \cite{Na1, Na2} about Cremona transformations, and important properties of plane cubic curves.  In this paper, we construct many more examples of positive entropy automorphisms on rational surfaces.  Whereas \cite{mcm} seeks automorphisms with essentially arbitrary topological behavior, we limit our search to automorphisms that might conceivably be induced by Cremona transformations of polynomial degree two (\emph{quadratic transformations} for short).  This restriction allows us be more explicit about the automorphisms we find and to make do with less technology, using only the group law for cubic curves (suitably interpreted when the curve is singular or reducible) in place of Coxeter theory and Nagata's theorems.

A quadratic transformation $f:\cp^2\to\cp^2$ always acts by blowing up three (\emph{indeterminacy}) points $I(f) = \{p_1^+,p_2^+,p_3^+\}$ in $\cp^2$ and blowing down the (\emph{exceptional}) lines joining them.  Typically, the points and the lines are distinct, but in general they can occur with multiplicity (see \S 1.2).  Regardless, $f^{-1}$ is also a quadratic transformation and $I(f^{-1}) = \{p_1^-,p_2^-,p_3^-\}$ consists of the images of the three exceptional lines.

Under certain fairly checkable circumstances a quadratic transformation $f$ will lift to an automorphism of some rational surface $X$ obtained from $\cp^2$ by a finite sequence of point blowups.  Namely, suppose there are integers $n_1,n_2,n_3\in\N$ and a permutation $\sigma\in\Sigma_3$ such that $f^{n_j-1}(p_j^-) = p_{\sigma_j}^+$ for $j=1,2,3$.  Then\footnote{We assume here that the $n_j$ are taken to be minimal.  To keep the present discussion simple we also assume that $f^k(p_j^-)\neq f^\ell(p_i^-)$ for any $k,\ell\geq 0$ and $i\neq j$.  We do not make the latter assumption outside this paragraph.  See \S 2.1 for a more complete discussion.} we can in effect cancel all indeterminate and exceptional behavior of $f$ by blowing up the finite sequences $p_j^-,f(p_j^-),\dots,f^{n_j-1}(p_j^-)$.  That is, if $X$ is the rational surface that results from blowing up these segments, then $f$ lifts to an automorphism $\hat f:X\to X$.  General theorems of Gromov \cite{gro} and Yomdin \cite{yom} imply directly that the entropy of this automorphism is $\log\lambda_1$, where the \emph{first dynamical degree} $\lambda_1$ is the spectral radius of the induced pullback operator $\hat f^*$ on $H^2(X,\R)$.   

Bedford and Kim observed (see the discussion surrounding Proposition \ref{fstar}) that the action $\hat f^*$ is entirely determined by $n_1,n_2,n_3$ and $\sigma$.  Hence we say that $\det(\hat f^*-\lambda\,\id)$ is the characteristic polynomial `for the \emph{orbit data} $n_1,n_2,n_3,\sigma$.'  When the $n_j$ are large enough (see \cite[Theorem 5.1]{bk2}), e.g. $n_j\geq 3$ with strict inequality for at least one $j\in\{1,2,3\}$, the characteristic polynomial has a root outside the unit disk, and hence $\hat f$ has positive entropy.  

Accordingly, one way to find positive entropy automorphisms induced by quadratic transformations would be to begin with some fixed quadratic transformation $q$, e.g. $q(x,y) = (1/x,1/y)$, and look for $T\in\Aut(\cp^2)$ such that $f=T\circ q$ \emph{realizes} the orbit data $n_1,n_2,n_3,\sigma$; i.e. so that $f^{n_j-1}(p_j^-) = p_{\sigma_j}^+$ for $j=1,2,3$.  This imposes essentially six conditions on $f$, so it seems plausible that some $T$ in the eight parameter family $\Aut(\cp^2)$ will serve.  However, the degrees of the equations governing $T$ increase exponentially with the $n_j$, and it therefore seems daunting to try to understand their solutions directly. 

A key idea in \cite{mcm}, which we follow here, is to look only at Cremona transformations $f$ that preserve some fixed cubic curve $C$.  Various aspects of such transformations have been studied in several recent papers (e.g. \cite{djs, pan1, pan2, bpv}).  We say that $f$ \emph{properly fixes} $C$ if $f(C) = C$ and no singular point of $C$ is indeterminate for $f$ or $f^{-1}$.  Then $f$ preserves both regular
and singular points $C_{reg},C_{sing}\subset C$ separately, and degree considerations imply that $I(f),I(f^{-1})\subset C$. As a Riemann surface, each connected component of $C_{reg}$ is equivalent to $\C/\Gamma$ for some (possibly rank 0 or 1) lattice $\Gamma\subset\C$.  The equivalence is not uniquely determined, and we assume it is chosen in a geometrically meaningful way; i.e. so that the conclusion of Theorem \ref{group law} below applies.  Under this equivalence, we have that the restriction of $f$ to any component of $C_{reg}$ is covered by an affine transformation $z\mapsto az+b$ of $\C$, with \emph{multiplier} $a\in\C^*$ satisfying $a\Gamma=\Gamma$.  Theorem \ref{quadauts} describes the prevalence and nature of the quadratic transformations that properly fix a given cubic $C$.  For $C$ irreducible, the theorem can be stated as follows.

\begin{theo}
\label{mainthm1}
Let $C\subset \cp^2$ be an irreducible cubic curve.  Suppose we are given points $p_1^+,p_2^+,p_3^+\in C_{reg}$, a multiplier $a\in\C^*$, and a translation $b\in C_{reg}$.  Then there exists at most one quadratic transformation $f$ properly fixing $C$ with $I(f) = \{p_1^+,p_2^+,p_3^+\}$ and $f|_{C_{reg}}:z\mapsto az+b$.  This $f$ exists if and only if the following hold.
\begin{itemize}
 \item $p_1^+ + p_2^+ + p_3^+ \neq 0$;
 \item $a$ is a multiplier for $C_{reg}$;
 \item $a(p_1^+ + p_2^+ + p_3^+) = 3b$;
\end{itemize}
Finally, the points of indeterminacy for $f^{-1}$ are given by $p_j^-= ap_j^+ - 2b$, $j=1,2,3$.
\end{theo}

Addition in the hypotheses and conclusions of this theorem depends on our identification of $\creg$ with the group $(\C/\Gamma,+)$.  The condition $\sum p_j^+\neq 0$ is equivalent to saying that $I(f)$ is not equal to the divisor obtained by intersecting $C$ with a line.  The third item constrains the translation $b$ for $f|_{C_{reg}}$ up to addition of an inflection point on $C_{reg}$.  It should be pointed out that the ideas underlying Theorem \ref{mainthm1} are not especially new.  Indeed, something similar to this theorem was used by Penrose and Smith \cite{ps} to better understand a restricted version of the family studied in \cite{bk}.

Here we apply Theorem \ref{mainthm1} to study quadratic transformations that fix each of the three basic types of irreducible cubic, and to identify those transformations that lift to automorphisms on some blowup of $\cp^2$.  Our first conclusion is

\begin{theo}
\label{mainthm3}
Let $n_1,n_2,n_3\in\N$, $\sigma\in\Sigma_3$ be orbit data whose characteristic polynomial has a root outside the unit circle.  Suppose that $C$ is an irreducible cubic curve and $f$ is a quadratic transformation that properly fixes $C$ and realizes the orbit data.  Then $C$ is one of the following.
\begin{itemize}
 \item The cuspidal cubic $y=x^3$. 
 \item A torus $\C/\Gamma$ with $\Gamma = \Z+i\Z$ or $\Gamma=\Z+e^{2\pi i/6}\Z$.
\end{itemize}
Both cases occur, but only finitely many sets of orbit data can be realized in the second one.
\end{theo}

When $C$ is a torus, the multiplier of the restriction $f|_C$ is necessarily a root of unity.  The problem with the nodal cubic and tori without additional symmetries is that the multiplier of a realization must be $\pm 1$, which implies (see Corollary \ref{minus1} and Theorem \ref{plus1}) that all roots of the characteristic polynomial lie on the unit circle.  In the case of tori with square or hexagonal symmetries, where multipliers can be $i$ or $e^{\pi i/3}$, one does get realizations lifting to automorphisms with positive entropy.  An interesting feature of these examples is that by passing to a fourth or sixth iterate, one obtains a positive entropy automorphism of a rational surface $X$ that nevertheless fixes the original cubic curve $C$ \emph{pointwise}.  We note that the group of Cremona transformations fixing a cubic was considered in \cite{bla}.

In general, realizations of orbit data by transformations whose multipliers are roots of unity seem to be somewhat sporadic, and we do not know how to characterize them systematically.  We have a better understanding when the multiplier is not a root of unity.

\begin{theo}
\label{mainthm2}
Suppose in the previous theorem that the multiplier $a$ of $f|_{C_{reg}}$ is not a root of unity.  Then
\begin{enumerate}
 \item $C$ is cuspidal;
 \item $a$ is a root of the characteristic polynomial for the given orbit data;
 \item if $n_1=n_2=n_3$, then $\sigma$ is the identity.
 \item if $n_i=n_j$ for $i\neq j$, then $\sigma$ does not interchange $i$ and $j$.
\end{enumerate}
Conversely, when these conditions are met by $C$ and $a$, there is a quadratic transformation $f$, unique up to conjugacy by a linear transformation fixing $C$, such that $f$ realizes the given orbit data, properly fixes $C$
and has multiplier $a$ on $C_{reg}$. Consequently, $f$ lifts to an automorphism on some rational surface $\pi:X\to\cp^2$ whose entropy is $\log\lambda_1$, where $\lambda_1>1$ is Galois conjugate to $a$.
\end{theo}

This result is reminiscent of those proved in \S 7 of \cite{mcm}.  In particular, the special cases discussed in \S 11 of that paper are included here.  These fix a cusp cubic and realize orbit data of the form $n_1=n_2=1$, $n_3\geq 8$, with $\sigma$ cyclic.  On the other hand, some of the maps in Theorem \ref{mainthm2} do not appear \cite{mcm}.  For instance, when $n_1=n_2=n_3\geq 4$, $\sigma =\id$, $I(f)$ degenerates to a single point, which is not permitted in McMullen's analysis.  To use the terminology from \cite{mcm}, coincidence of two points in $I(f)$ implies the existence of a `geometric nodal root' for the action $\hat f^*$ of the induced automorphism.

We also consider quadratic transformations fixing reducible cubics $C$, relying on the more general version Theorem \ref{quadauts} of Theorem \ref{mainthm1}.  If $C$ is reducible with one singularity, then things turn out much as they did for the cuspidal cubic.  The arguments used to prove Theorem \ref{mainthm2} remain valid once one accounts for the facts that $f$ permutes the components of $C_{reg}$ and that this permutation must be compatible with the one prescribed in the given orbit data. The end result (Theorem \ref{concurrent lines}) is that one can realize somewhat fewer, though still infinitely many, different sets of orbit data.

If $C$ has two or three singular points, things turn out differently.  Any quadratic transformation $f$ that properly fixes $C$ must have multiplier $f|_{C_{reg}}$ equal to $\pm 1$.  Nevertheless, by judiciously choosing the translations for $f|_{C_{reg}}$ we are still able to realize infinitely many sets of orbit data.  We treat the case $\#C_{sing} = 3$ more thoroughly (see Theorem \ref{actual3}).

\begin{theo}
\label{mainthm4}
Let $n_1,n_2,n_3\geq 1$ and $\sigma\in\Sigma_3$ be orbit data whose characteristic polynomial has a root outside the unit circle.  If the orbit data is realized by some quadratic transformation $f$ that properly fixes $C=\{xyz=0\}$, then $\sigma = \id$, and $f$ maps each component of $C_{reg}$ to itself with multiplier $1$.
Conversely, when $\sigma = \id$ and $n_1,n_2,n_3\geq 6$, there exists at least one such realization.
\end{theo}
The proof amounts to an extended exercise in arithmetic mod 1.  Unlike Theorem \ref{mainthm2}, the conclusion gives  little idea of how many different realizations are possible.  We simply show that for any given orbit data, there are finitely many quadratic transformations that might serve as realizations, and then we find one candidate from among these that works.

We deal more briefly with the case where $C$ has two irreducible components meeting transversely, i.e. $C=\{(xy-z^2)z=0\}$,  showing that one can realize only two broad types of orbit data on this curve and then giving examples of each type.

The remainder of the paper is organized as follows.  \S 1 provides background on plane cubics and quadratic transformations, culminating in the proof of Theorem \ref{quadauts}. \S 2 begins by considering when and how a quadratic transformation can be lifted to an automorphism $\hat f:\hat X\self$. It then discusses the nature of the associated operator $\hat f^*:H^2(X,\R)\to H^2(X,\R)$, which can be written down very explicitly and fairly simply in terms of the given orbit data.  In \S 3 we seek automorphisms induced by quadratic transformations that properly fix irreducible cubics, and in \S 4 we treat the reducible case.  The Appendix to this paper, which was contributed by Igor Dolgachev, gives a detailed treatment of the group law on reduced plane cubics that includes the case of singular and reducible curves.

First and foremost, we would like to thank Igor Dolgachev for his extensive help concerning the geometry of plane cubics.
We would also like to thank Eric Riedl, Kyounghee Kim, and Eric Bedford for their comments and attention as this paper was written.  

\section{Quadratic transformations fixing a cubic}

In this section, we recount some well-known facts about cubic curves and quadratic Cremona transformations in the plane.  Then we characterize those quadratic transformations that `properly' fix a given cubic.  We refer the reader to the recent article \cite{CeDe} for more discussion of quadratic transformations.

\subsection{The `group law' on plane cubics}
Let $C\subset \cp^2$ be a \emph{cubic curve}; that is, $C$ is defined by a degree three homogeneous polynomial without repeated factors.
Hence $C$ has at most three irreducible components $V\subset C_{reg}$, each isomorphic after normalization to either a torus (when $C$ is irreducible and smooth) or $\cp^1$ (in all other cases).  We begin by recalling some facts that are discussed at greater length in the Appendix.

The Picard group $\Pic(C)$ consists of linear equivalence classes $[D]$ of Cartier divisors $D$ on $C$, and the subgroup $\pico(C)\subset \Pic(C)$ consists of divisor classes whose restrictions to each irreducible component have degree zero.  In fact, one always has $\pico(C) \cong \C/\Gamma$ where $\Gamma\subset\C$ is a lattice of rank $2$, $1$, or $0$ depending on whether $C$ has no singularities, nodal singularities, or otherwise.  Moreover, for any irreducible $V\subset C$ and any choice of `origin' $0_V\in V$, one has a bijection $\kappa:V\cap C_{reg} \to \pico(C)$ given by $\kappa(p) = [p - 0_V]$ that allows us to regard the smooth points $V\cap C_{reg}$ in $V$ as a group isomorphic to $\pico(C)$.  We will always use $+$ to denote the group operation, even when $\Gamma\cong \Z$ has rank 1 and $\pico(C) \cong \C^*$.  

Having fixed origins in each irreducible component of $C$, we will write $p_1 + p_2\sim p_3$ for any $p_1,p_2,p_3\in C_{reg}$ to mean that $\kappa(p_1) + \kappa(p_2) = \kappa(p_3)$; in other words the $\sim$ implies that any point $p\in C_{reg}$ that appears in the equation is implicitly identified with the point $\kappa(p)\in \pico(C)$.  Note that we do not require $p_1,p_2,p_3$ to lie on the same irreducible component of $C$, even though we have not given $C_{reg}$ itself the structure of a group (this can be done; see the Appendix).  We further caution that with our convention, `$\sim$' does not denote linear equivalence. In fact, since the choice of origins $0_V$ is a priori arbitrary, the equation $p_1+p_2\sim p_3$ need not have much geometric content at all.  To make such equations more meaningful, we will assume that the origins are chosen to satisfy
\begin{itemize}
 \item[($*$)\hspace*{.2in}] $\sum_{V\subset C} (\deg V)\cdot 0_V$ is the divisor cut out by a line $L_0\subset\cp^2$.
\end{itemize}
This condition guarantees that three points $p_1,p_2,p_3 \subset C_{reg}$ are the intersection (with multiplicity) of $C$ with a line $L\subset\cp^2$ if and only if each irreducible $V\subset C$ contains $\deg V$ of the points, and $x+y+z\sim 0$.  More generally, we have the following classical fact.

\begin{thm}
\label{group law}
Suppose that the projection $\kappa:C_{reg}\to \pico(C)$ is chosen to satisfy ($*$).  Then $3d$ (not necessarily distinct) points $p_1,\dots, p_{3d}\in C_{reg}$ comprise the intersection of $C$ with a curve of degree $d$ if 
and only if 
\begin{itemize}
 \item each irreducible $V\subset C$ contains $d\cdot\deg V$ of the points; and 
 \item $\sum p_j \sim 0$.
\end{itemize}
\end{thm}

Before continuing, let us quickly recapitulate this discussion in more analytic terms: the various connected components $V\cap C_{reg}$ of $C_{reg}$ are all isomorphic as Riemann surfaces to the same surface $\C/\Gamma$.
These isomorphisms are determined only up to affine transformations, but they may be chosen so that for two lines $L_0,L_1\subset C_{reg}$ we have that the three points (counted with multiplicity) $L_j\cap C$ all lie in $C_{reg}$ and are identified with three points summing to zero in $\C/\Gamma$.  This choice being equivalent to condition ($*$), it follows that Theorem \ref{group law} holds.  

In all cases except that of a smooth cubic whose single irreducible component is not rational, the projection $\kappa:C_{reg}\to \C/\Gamma\cong\pico(C)$ can be written down quite explicitly.  For instance, when $C$ is a cusp cubic, then $\pico(C) \cong \C$ and we can choose coordinates on $\cp^2$ so that $C_{reg} = \{y=x^3:x\in\C\}$.  We then define $\kappa(x,x^3) = x$.  Or if $C$ is a union of a conic and a secant line, then $\pico(C) \cong \C^*$ and we may assume $C=\{z(z^2-xy)=0\}$.  A suitable projection is then given by mapping $[1:-t:0], [t^2:1:t]\mapsto t$ for all $t\in\C^*$.  

We will say that $T\in\Aut(\cp^2)$ \emph{fixes} (or \emph{leaves invariant}) $C$ if $T(C) = C$ as sets.  That is, $T$ restricts to an automorphism of $C$ and therefore induces a map $T^*:\pic(C)\to \pic(C)$ whose restriction
to $\pico(C)$ is group automorphism given by $t\in \C/\Gamma\mapsto a^{-1}t$ for some \emph{multiplier} $a\in\C^*$ satisfying $a\Gamma = \Gamma$.  Explicitly, the possible multipliers $a\in\C^*$ are as follows.
\begin{itemize} 
 \item if $C$ smooth and irreducible, $a=\pm 1$ generically, but $a=i^k$ when $C = \C/(\Z+i\Z)$, and $a=e^{\pm\pi ik/3}$ when $C = \C/(\Z+e^{\pi i/3}\Z)$;
 \item if $C$ has nodal singularities, $a=\pm 1$;
 \item in all other cases, arbitrary any $a\in\C^*$ is possible. 
\end{itemize}
Now if $V\subset C$ is any irreducible component, and $p\in V$, then $[T(p) - 0_{T(V)}] = [T(p) - T(0_V)] + [T(0_V) - 0_{T(V)}] = a[p-0_{T(V)}] + b_V$, where $a$ is the multiplier corresponding to $T^*$ and 
$b_V := [T(0_V) - 0_{T(V)}] \in \pico(C)$ is the \emph{translation} for $T|_V$.  More succinctly, using our convention above, we have that $T:V\to T(V)$ is an `affine transformation' described by $T(p) \sim ap + b_V$ for all $p\in V$.  Since $T$ sends lines to lines and we are assuming that condition ($*$) holds, it follows that $\sum_{V\subset C} (\deg V)\cdot T(0_V) \sim 0$.  Indeed it is shown in the Appendix (Corollary \ref{projauts}) that 

\begin{prop}
\label{fixingauts}
Let $T\in\Aut(\cp^2)$ be a linear transformation fixing a cubic curve $C$.  Then the translations $b_V$ for the restrictions $T|_V$ of $T$ to the various irreducible components $V\subset C$ satisfy 
$\sum (\deg V)\cdot b_V \sim 0$.  Conversely, given translations subject to this condition and a multiplier $a$ for $\pico(C)$, there exists a unique $T\in\Aut(\cp^2)$ fixing each component $V\subset C$ with multiplier $a$ and translations $b_V$.
\end{prop}

When $C$ is irreducible, the condition on the translations may be stated more simply by saying that the translation corresponds to an inflection point of $C_{reg}$.  When the cubic $C$ is union of three lines, then it is easy to find automorphisms of $\cp^2$ that permute the lines arbitrarily.  Therefore in this case, the transformation $T$ in the final statement of the theorem can alternately be chosen to permute the lines in any desired fashion.

\subsection{Quadratic Cremona transformations}

The most basic non-linear Cremona (i.e. birational) transformation $q:\cp^2\to\cp^2$ can be expressed in homogeneous coordinates as $[x:y:z]\mapsto [yz:zx:xy]$.  Geometrically, $q$ acts by blowing up the points $[0:0:1]$, $[0:1:0]$, $[1:0:0]$ and then collapsing the lines $\{x=0\}$, $\{y=0\}$, $\{z=0\}$ that join them.  A generic quadratic Cremona transformation can be obtained from $q$ by pre- and post- composing with linear transformations $f=L\circ q \circ L'$.

In fact, every quadratic transformation (we henceforth omit the word `Cremona') $f$ can be obtained geometrically by blowing up three points $p_1^+,p_2^+,p_3^+$ and collapsing three rational curves.  We call the $p_j^+$ \emph{indeterminacy points} (alternately \emph{base points}, or \emph{fundamental points}) for $f$ and let $I(f)$ denote the set they comprise.  We call the contracted curves \emph{exceptional} for $f$.  If $f$ is a quadratic transformation, then so is $f^{-1}$, and we have $I(f^{-1}) = \{p_1^-,p_2^-,p_3^-\}$, where each $p_j^-$ is the image of one of the exceptional curves for $f$.  The indices $1,2,3$ assigned to points in $I(f)$ naturally determine an indexing of the points in $I(f^{-1})$  In the situation of the previous paragraph, this is given by declaring $p_j^-$ to be the image of the exceptional line that does not contain $p_j^+$.  In the sequel, however, we must allow our quadratic transformations to be degenerate, so we briefly review the three possibilities for the geometry of a quadratic transformation $f:\cp^2\to\cp^2$.
\begin{itemize}
 \item \emph{Generic case.} The points $p_1^+,p_2^+,p_3^+ \in \cp^2$ are distinct.  They are all blown up (in any order) and the lines joining them are then contracted.
\item \emph{Generic degenerate case.}  We have $p_i^+ = p_j^+\neq p_k^+$ for $\{i,j,k\} = \{1,2,3\}$.  In this case,
there is an exceptional line $E_j^-$, joining $p_i^+$ and $p_k^+$, and another exceptional line $E_k^-$ containing $p_i^+$.  First, $f$ blows up $p_i^+$ and $p_k^+$, creating new rational curves $E_i$ and $E_k^+$.  Then
$f$ blows up the point $E_k^-\cap E_i$ (which lies over $p_j^+$).  Next $f$ contracts $E_j^-$; finally $f$ contracts $E_i$ and $E_k^-$.
\item \emph{Degenerate degenerate case.} We have $p_1^+=p_2^+=p_3^+$.  There is a single exceptional line $E_k^-\subset \cp^2$ containing $p_i^+$.  The transformation blows up $p_i^+$ creating a curve $E_i$, then blows 
up $E_k^-\cap E_i$ creating $E_j$, and finally blows up some point on $E_j$ different from $E_j\cap E_k^-$ creating a curve $E_k^+$; to descend back to $\cp^2$, $f$ contracts $E_k^-$, $E_j$ and $E_i$ in order.
\end{itemize}
In the degenerate cases, we will readily abuse notation by treating e.g. $p_k^+$ as a point in $\cp^2$ and also identifying it with the infinitely near point that is blown up to create $E_k^+$.  In the first sense $I(f)$ contains no more than three points, but in the second sense it always contains exactly three.  The important thing is that in either sense, the points in $I(f^{-1})$ are indexed so that $p_k^-$ is the image of $E_k^-$ after contraction.  We note also that in each of the three cases, the geometry of $f$ and $f^{-1}$ is the same, so that
$p_j^+$ is infinitely near to $p_i^+$ if and only if $p_j^-$ is infinitely near to $p_i^-$, and $\#I(f) = \#I(f^{-1})$ as sets in $\cp^2$. In order to avoid tedious case-by-case exposition in this paper, we will generally give complete arguments only for the generic case where the points $p_j^+$ are distinct, attending to details of the other cases only when they are conceptually different.

Given a curve $C\subset\cp^2$ and a quadratic transformation $f$, we define $f(C) \eqdef \overline{f(C\setminus I(f))}$ to be the proper transform of $C$ by $f$.  When $C\cap I(f) = \emptyset$, we have that $\deg f(C) = 2\deg C$.  In general 
\begin{equation}
\label{deg eqn}
\deg f(C) = 2\deg C - \sum_{p\in I(f)} \nu_p(C),
\end{equation} 
where $\nu_p(C)$ is the multiplicity of $C$ at $p$.  Note that if $p$ is infinitely near, appearing only in some modification $\pi:X\to\cp^2$, then we take $\nu_p(C)$ to be the multiplicity at $p$ of the proper transform of 
$C$ by $\pi^{-1}$.

We will say that $C$ is \emph{fixed} or \emph{invariant} by $f$ if $f(C) = C$. We will further say that $C$ is \emph{properly fixed} by $f$ if additionally all points in $I(f)\cap C$ and $I(f^{-1})\cap C$ are regular for $C$.  In this case, we have that $f$ permutes the singular points of $C$, preserves their type and restricts to a well-defined automorphism of $C$.  Now suppose $C$ is a cubic curve.  As we discussed prior to Proposition 
\ref{fixingauts}, the automorphism $f|_C$ can be described by the multiplier $a\in\C^*$ for the action $(f|_C)^*:\pico(C)\to\pico(C)$, the way it permutes the irreducible components $V\subset C$, and the translations $b_V = [f(0_V) - 0_{f(V)}]\in\pico(C)$ for each of these components.  We note that unlike the situation with projective automorphisms, one can have $\deg f(V)\neq \deg V$ for an irreducible component of $V$.  The starting point for our work is the following detailed description of the quadratic transformations properly fixing a given cubic.

\begin{thm} 
\label{quadauts}
Let $\tau:C\to C$ be an automorphism with multiplier $a$ and translations $b_V$, $V\subset C$.  Given points $p_1^+,p_2^+,p_3^+\in \cp^2$, there exists a quadratic transformation $f:\cp^2\to\cp^2$ properly fixing $C$ with $f|_C = \tau$ if and only if
\begin{enumerate}
 \item 
\label{1st} 
For each irreducible $V\subset C$, we have $\# \{j:p_j^+\in V\cap C_{reg}\} = 2\deg V - \deg \tau(V)$ and $\# \{j:p_j^- \in V\} = 2\deg V - \deg \tau^{-1}(V)$.  In particular $I(f)\subset C_{reg}$.
 \item 
\label{3rd}
$\sum p_j^+ \sim a^{-1}\sum_{V\subset C} (\deg V)\cdot b_V \neq 0$.
\end{enumerate}
The transformation $f$ is unique when it exists and the points of indeterminacy $p_j^- \in I(f^{-1})$ then satisfy the following.
\begin{enumerate}
\setcounter{enumi}{2}
\item
\label{2nd}
Given $j\in\{1,2,3\}$, let $L$ be the line defined by the two points $I(f)\setminus\{p_j^+\}$, and let $V\subset C$ be the irreducible component containing the third point in $C\cap L$.  Then $p_j^- \in \tau(V)$.
\item 
\label{4th}
For each $j\in \{1,2,3\}$, we have $p_j^- - ap_j^+ \sim b_j - \sum b_V\deg V$, where $b_j$ is the translation for $f$ on the component containing $p_j^+$.
\end{enumerate}
\end{thm}

\begin{proof}
Suppose first that there exists a quadratic transformation $f$ with the desired properties; i.e. $f$ properly fixes $C$ with $f|_C = \tau$ and $I(f) = \{p_1^+,p_2^+,p_3^+\}$.  Condition \ref{1st} is then a consequence of the degree equation \eqref{deg eqn}.  Condition \ref{2nd} follows from the relationship, described at the beginning of this subsection, between points in $I(f)$ and points in $I(f^{-1})$.  

To see that condition \ref{3rd} holds, note first that since the $p_j^+$ are indeterminate for $f$, they cannot be collinear.  Hence $\sum p_j^+ \not\sim 0$ by Theorem \ref{group law}.  let $L\subset\cp^2$ be a generic line.  Then by Theorem \ref{group law}, we have $p_1+p_2+p_3\sim 0$ where $p_1,p_2,p_3\in C_{reg}$ are the points where $L$ meets $C$.  Now $f^{-1}(L)$ 
is a conic containing the three points $f^{-1}(p_j)$ and (since $L$ meets all exceptional lines for $f^{-1}$ in generic points) the three points in $I(f)$.  Thus $\sum_{j=1}^3 f^{-1}(p_j) + \sum_{j=1}^3 p_j^+ \sim 0$.  
Moreover, since $\tau(p) \sim ap + b_V$ for all $p$ in an irreducible component $V\subset C$, we see that
$
f^{-1}(p_j) = \tau^{-1}(p_j) \sim a^{-1}(p_j - b_j),
$ 
where $b_j$ is the translation for the irreducible component $V\subset C$ containing $\tau^{-1}(p_j)$.  Each such $V$ contains $\deg V$ of the points $p_j$, so we infer
$$
0 \sim \sum p_j^+ + a^{-1}\sum (p_j - b_j) \sim \sum p_j^+ - a^{-1}\sum_{V\subset C} (\deg V) b_V.
$$
as desired.

Condition \ref{4th} follows from the same kind of reasoning.  Taking $j=1$, we let $L$ be a generic line passing through $p_1^+$, and let $p_2,p_3\in C$ be the remaining points on $L\cap C$.  Then we have 
$p_1^+ + p_2 + p_3 \sim 0$.  By \eqref{deg eqn}, the image $f(L)$ is also a line.  Clearly $L$ contains $f(p_j)\sim ap_j + b_j$ for $j=2,3$, where this time $b_j$ is translation for the irreducible component containing $p_j$.  Also, $L$ intersects the exceptional line through $p_2^+$ and $p_3^+$ at a generic point, so $p_1^-\in f(L)$.  Hence $p_1^-+f(p)+f(q)\sim 0$.  We combine this information to get
$$
0 \sim p_1^- + a(p_2+p_3) + b_2 + b_3 \sim p_1^- - ap_1^+ - b_1 + \sum_{V\subset C} b_V,
$$
where the last $\sim$ follows from the fact that $L$ intersects each irreducible component $V\subset C$ in $\deg V$ points.  So condition \ref{4th} holds.  In summary, conditions \ref{1st} through \ref{4th} are necessary for existence of $f$.  

Turning to sufficiency, we suppose rather that the given automorphism $\tau$ and the points $p_j^+$ satisfy conditions \ref{1st} and \ref{3rd}.  The points $p_j^+$ are not collinear by condition \ref{3rd} and Theorem \ref{group law}, so there exists a quadratic transformation $f$ with $I(f) = \{p_1^+,p_2^+,p_3^+\}$.  It follows from the degree equation \eqref{deg eqn} that $f(C)$ is a cubic curve isomorphic to $C$.  Therefore $f(C) = T(C)$ for some $T\in\Aut(\cp^2)$.  Replacing $f$ with $T^{-1}\circ f$, we have that $f$ properly fixes $C$.  Further composing with a planar automorphisms that permutes linear components of $C$, we may assume that $f(V) = \tau(V)$ for each irreducible $V\subset C$.   

Let $\tilde a\in\C^*$ be the multiplier for the induced automorphism $f|_C$.  Multipliers for the curve $C$ form a group, so from Theorem \ref{fixingauts} we obtain $S\in\Aut(\cp^2)$ fixing $C$ component-wise such that $S(p) \sim a\tilde a^{-1}p$ for all $p\in C$.  We replace $f$ with $S\circ f$ to get $\tilde a = a$.  By the first part of the proof the translations $\tilde b_V$ for $f|_C$ satisfy condition \ref{3rd}.  In particular, 
$\sum (b_V - \tilde b_V) = 0$.  Applying Theorem \ref{fixingauts} again, we get $R\in \Aut(\cp^2)$ fixing $C$ component-wise and satisfying $R(p) \sim p + (b_V-\tilde b_V)$ for each irreducible $V\subset C$ and all 
$p\in f(V)$.  Trading $f$ for $R\circ f$, we arrive finally at a quadratic transformation with all the desired properties.

To see that this $f$ is unique, note that if $\tilde f$ is another such transformation, then $f\circ \tilde f^{-1}$ is a planar automorphism that fixes $C$ pointwise.  In particular, $f\circ \tilde f^{-1}$ fixes three distinct points on any generic line in $\cp^2$ and therefore fixes generic lines pointwise.  It follows that $f = \tilde f$.
\end{proof}

Let us close this section with a couple of remarks.  When applying Theorem \ref{quadauts}, one can of course, specify the points in $I(f^{-1})$ rather than those in $f$.  In this case, condition \ref{3rd} in the theorem becomes
$\sum p_j^- \sim -\sum (\deg V)\cdot b_V$, as one can see by summing condition \ref{4th} over $j=1,2,3$ and combining it with the version of condition \ref{3rd} appearing in the theorem.  

If the cubic $C$ is singular, then it is possible to write down algebraic formulas for the quadratic transformations $f$ in Theorem \ref{quadauts} (see \cite{jackson} for some of these).  However, these tend to be quite long, and it seems to us preferable in many instances to take a more algorithmic point of view.  Namely, if $p\in\cp^2$ is a point outside $C$ and not lying on an exceptional curve, then for any $p_j^+\in I(f)$, the line $L$ joining $p$ and $p^+_j$ meets $C_{reg}$ in two more points $x$ and $y$.  Additionally, the exceptional line that maps to $p_j^-$ meets $L$ in a point $q$.  The image $f(L)$ is therefore also a line, and it passes through $f(x)$, $f(y)$, and $f(q) = p_j^-$.  These last three points are determined by $I(f)$ and $f|_C$.  So we can find $f(L)$ explicitly.  Since $f|_L:L\to f(L)$ is a map between copies of $\cp^1$, and we know the images of three distinct points under $f|_L$, we can find an explicit formula for $f|_L$ and in particular for $f(p)$.  

\section{Automorphisms from quadratic transformations}
\label{auts}

In this section, we consider the issue of when and how a quadratic transformation will lift to an automorphism on some blowup of $\cp^2$.  We also consider the linear pullback actions induced by such automorphisms.  Several of the results here are assembled from other places and restated in a form that will be convenient for us.

\subsection{Lifting to automorphisms}
Let us first describe the precise situation and manner in which a quadratic transformation $f$ can be lifted to an automorphism on a rational surface $X$ obtained from $\cp^2$ be a sequence of blowups (see \cite{bk} and \cite{df} for more on this).  Suppose that there exists $n_1\in\N$ and $\sigma_1\in\{1,2,3\}$ such that $f^{n_1-1}(p_1^-)= p_{\sigma_1}^+$.  Relabeling the points $p_j^-$ and changing the index $\sigma_1$ if necessary, we may further assume that 
\begin{itemize}
 \item $n_1$ is minimal, i.e. $f^j(p_1^-)\notin I(f)\cap I(f^{-1})$ for any $0<j<n_1-1$, and $p_1^-\in I(f)$ only if $n_1=1$;
 \item $p_1^-$ is not infinitely near to some other point in $I(f^{-1})$;
 \item $p_{\sigma_1}^+$ is not infinitely near to some other point in $I(f)$.
\end{itemize}
Then by blowing up the points $p_1^-,\dots, f^{n_1-1}(p_1^-)$, we obtain a rational surface $X_1$ to which $f$ lifts as a birational map $f_1:X_1\to X_1$ with only two points (counting multiplicity) $p_2^-,p_3^-\in I(f_1^{-1})$. If then $f_1^{n_2-1}(p_2^-) = p_{\sigma_2}^+$ for some $n_2\in\N$ and $\sigma_2\neq \sigma_1$, then we can repeat this process obtaining a map $f_2:X_2\to X_2$ with only one point $p_3^-\in I(f_2^{-1})$.  If finally $f_2^{n_3-1}(p_3^-) = p_{\sigma_3}^+$, then we blow up along this last orbit segment and arrive at an automorphism $\hat f:X\to X$.  We call the integers $n_1,n_2,n_3\geq 1$ together with the permutation $\sigma\in\Sigma_3$ the \emph{orbit data} associated to $f$, noting that the surface $X$ is completely determined by the orbit data and the points $p_j^-\in I(f^{-1})$.  Conversely, we say that the quadratic transformation $f$ \emph{realizes} the orbit data $n_1,n_2,n_3,\sigma$.
It follows from general theorems of Yomdin and Gromov (see e.g. \cite{can}) that the topological entropy of any automorphism $\hat f:X\to X$ of a rational surface $X$ is $\log\lambda_1$, where $\lambda_1$ is the largest eigenvalue of the induced linear operator $\hat f^*:H^2(X,\R)\to H^2(X,\R)$.  If $\hat f$ is the lift of a quadratic transformation as in the previous paragraph, then it is not difficult to describe $\hat f^*$ explicitly.  Let $H\in H^2(X,\R)$ be the pullback to $X$ of the class of a generic line in $\cp^2$.  Let $E_{i,n}\in H^2(X)$, $0\leq n\leq n_i-1$ be the class of the exceptional divisor\footnote{Note that this will sometimes be reducible if there are infinitely near points blown up in constructing $X$} associated to the blowup of $f^{n}(p_i^-)$.  Then $H$ and the $E_{i,n}$ give a basis for $H^2(X,\R)$ that is orthogonal with respect to intersection and normalized by $H^2 =1$, $E_{i,n}^2 = -1$.  Under $\hat f^*$ we have
\begin{eqnarray*}
H & \mapsto & 2H-E_{1,n_1-1}-E_{2,n_2-1}-E_{3,n_3-1} \\
E_{i,n} & \mapsto &  E_{i,n-1}, \text{ for } 1\leq n\leq n_i-1;
\end{eqnarray*}
and under $\hat f_* = (\hat f^*)^{-1}$ we have
\begin{eqnarray*}
H & \mapsto & 2H-E_{1,0}-E_{2,0}-E_{3,0} \\
E_{i,n-1} & \mapsto & E_{i,n}, \text{ for } 1\leq n\leq n_i-1;
\end{eqnarray*}
Hence we arrive at

\begin{prop}
\label{fstar}
With the notation above, we have $\hat f^* = S\circ Q$, where $Q:H^2(X)\to H^2(X)$ is given by 
$$
Q(H) = 2H - E_{1,0}-E_{2,0}-E_{3,0}, \quad Q(E_{i,0}) = H-\sum_{j\neq i} E_{j,0}, \quad Q(E_{i,n}) = E_{i,n} \text { for } n>0;
$$
and $S$ fixes $H$ and permutes the $E_{i,j}$ according to
$$
E_{\sigma_i,0}\mapsto E_{i,n_i-1}, \quad E_{i,n} \mapsto E_{i,n-1} \text{ for } n<n_i-1.
$$
The characteristic polynomial $P(\lambda)$ for $\hat f^*$ has at most one root outside the unit circle, and if it exists this root is real and positive.  Moreover, every root $\lambda = a$ of $P(\lambda)$ is Galois conjugate over $\Z$ to its reciprocal $a^{-1}$.
\end{prop}

\begin{proof}
The decomposition $\hat f^* = S\circ Q$ follows from the discussion above.  The assertion about roots outside the unit circle is well-known (see \cite{can}) and follows from the fact that the intersection form on $H^2(X,\R)$ has exactly one positive eigenvalue.  Now if $\lambda = e^{i\theta}$ is a root of $P(\lambda)$ on the unit circle, then $e^{i\theta}$ is Galois conjugate to $\overline{e^{i\theta}} = (e^{i\theta})^{-1}$ because $\hat f^*$ preserves integral cohomology classes.  And if $\lambda = a > 1$ is a root of $P(\lambda)$, then so is $a^{-1}$, because $\hat f^*$ and $\hat f_* = (\hat f^*)^{-1}$ are adjoint with respect to intersection, and therefore have the same characteristic polynomials.  Since the product of the roots of the minimal polynomial for $a^{-1}$ must be an integer, it follows that $a$ and $a^{-1}$ are Galois conjugate over $\Z$.
\end{proof}

Proposition \ref{fstar} implies that the action $\hat f^*$ (as well as the hyperbolic space $H^2(X,\R)$) depends only on the orbit data associated to $f$.  In fact, given any orbit data $n_1,n_2,n_3,\sigma$, whether or not it is realized by some quadratic transformation $f$, one can consider the (abstract) isometry
$$
\hat f^*:V \mapsto V
$$
of the hyperbolic z space $V=\R H\bigoplus_{ij} \R E_{ij}$ defined by the equations preceding Proposition \ref{fstar}, and the characteristic polynomial of this isometry will still satisfy the conclusions of the proposition.

We observe in passing that if $\sigma$ is the identity permutation, then the permutation $S$ in the theorem decomposes into three cycles
$$
S=(E_{1,n_1-1}\dots E_{1,0})(E_{1,n_2-1}\dots E_{2,0})(E_{3,n_3-1}\dots E_{3,0});
$$
if $\sigma$ is an involution, swapping e.g. $1$ and $2$, then $S$ decomposes into two cycles
$$
S = (E_{1,n_1-1}\dots E_{1,0}E_{2,n_2-1}\dots E_{2,0})(E_{3,n_3-1}\dots E_{3,0});
$$
and if $\sigma = (123)$ is cyclic, then $S$ is cyclic
$$
S= (E_{1,n_i-1}\dots E_{1,0}E_{2,n_1-1}\dots E_{2,0}E_{3,n_2-1}\dots E_{3,0}).
$$

Bedford and Kim \cite{bk} have computed $P(\lambda)$ explicitly for any orbit data $n_1,n_2,n_3,\sigma$, and their formula will be useful to us below (see the fortuitous coincidence in the proof of Theorem \ref{tentative}).  Specifically, they show that $P(\lambda) =\lambda^{1+\sum n_j}p(1/\lambda)+(-1)^{\operatorname{ord}\sigma}p(\lambda)$, where 
\begin{equation}
\label{charpoly}
p(\lambda) = 1 - 2\lambda + \sum_{j=\sigma_j} \lambda^{1+n_j} + \sum_{j\neq \sigma_j} \lambda^{n_j}(1-\lambda).
\end{equation}

\subsection{Some general observations}
The following fact is folklore among people working in complex dynamics.  We include the proof for the reader's convenience.

\begin{prop}
\label{toofew}
Let $X$ be a rational surface obtained by blowing up $n\leq 9$ points in $\cp^2$ and $f:X\to X$ be an automorphism.
Then the topological entropy of $f$ vanishes.  If $n\leq 8$, then $f^k$ descends to a linear map of $\cp^2$ for some $k\in\N$.
\end{prop}

\begin{proof}
Suppose that $f$ has positive entropy $\log\lambda>0$.  Then there exists \cite{can} a non-trivial real cohomology class $\theta\in H^2(X,\R)$ with $f^*\theta = \lambda\theta$ and $\theta^2 = 0$.  Moreover, $f_*K_X = f^*K_X = K_X$, where $K_X$ is the class of a canonical divisor on $X$.  Intersecting $K_X$ and $\theta$, we see that
$$
\pair{\theta}{K_X} = \pair{f^*\theta}{f^*K_X} = \pair{\lambda \theta}{K_X}.
$$
Hence $\pair{\theta}{K_X} = 0$.  Since the intersection form on $X$ has signature $(1,n-1)$, and $K_X^2\geq 0$ for $n\leq 9$ we infer that $\theta = cK_X$ for some $c<0$.  But then $f^*\theta = \theta \neq \lambda\theta$.
This contradiction shows that $f$ has zero entropy.

If $n\leq 8$, then in fact $K_X^2 > 0$.  Thus the intersection form is strictly negative on the orthogonal complement $H\subset H^2(X,\R)$ of $K_X$.  Since $H$ is finite dimensional and invariant under $f^*$, and $f^*$ preserves $H^2(X,\Z)$, it follows that $f^*$ has finite order on $H$.  Hence $f^{k*} = \id$ for
some $k\in\N$.  In particular, $f^k$ preserves each of the exceptional divisors in $X$ that correspond to the $n\leq 8$ points blown up in $\cp^2$.  It follows that $f^k$ descends to a well-defined automorphism of $\cp^2$.
\end{proof}

\begin{cor}
\label{minus1}
Suppose that $f:\cp^2\to\cp^2$ is a quadratic transformation that properly fixes a cubic curve $C\subset\cp^2$ and lifts to an automorphism $\hat f$ of some modification $X\to \cp^2$.  If the multiplier of $f|_C$ is $-1$ and $f$ fixes each irreducible component of $C$, then $f:\cp^2\to\cp^2$ is linear.  Similarly, if $f$ fixes each irreducible component of $C$ and the multiplier of $f|_C$ is a primitive cube root of unity, then the topological entropy of $\hat f$ vanishes.
\end{cor}

\begin{proof}
Suppose $f$ realizes orbit data $n_1,n_2,n_3\geq 1$ , $\sigma\in\Sigma_3$.  If the multiplier of $f$ is $-1$ and $f^2(V) = V$ for each irreducible $V\subset C$, then it follows that $f^2|_C = \id$.  Hence $n_j = 1$ or $2$ for each $j$, and the surface $X$ may be created by blowing up at most six points in $\cp^2$.  The first assertion follows from Proposition \ref{toofew}.  If the multiplier of $f$ is a primitive cube root of unity, then $f^3$ fixes $C$ component-wise, and the same argument shows that $X$ may be constructed by blowing up at most $9$ points in $\cp^2$.  The second assertion likewise follows.
\end{proof}

\begin{thm}
\label{plus1}
Let $f:\cp^2\to\cp^2$ be a quadratic transformation properly fixing a cubic curve $C\subset\cp^2$.  Suppose that $f$ permutes the irreducible components of $C$ transitively and that $f|_C$ has multiplier $1$. Let $X$ be the rational surface obtained by blowing up all points (with multiplicity) in $I(f)$, $I(f^{-1})$ and $f(I(f^{-1}))$. Then $f$ lifts to an automorphism $\hat f:X\to X$ with an invariant elliptic fibration.
\end{thm}

Of course, the topological entropy must vanish for the map in this theorem.  A more detailed analysis shows that either $f^2 = \id$, or $\norm{\hat f^{n*}}$ grows quadratically with $n$ and the invariant elliptic fibration is unique (see \cite{ps, can, mcm}) for more about this phenomenon.

\begin{proof} 
We claim that after conjugation by a planar automorphism, we may assume that the translations $b_V$ for $f$ on the irreducible components $V\subset C$ are independent of $V$.  To see this, suppose that $C$ has three irreducible components permuted $V_1 \to V_2 \to V_3 \to V_1$ by $f$.  Let $b_1$ be the corresponding translations.  Then choose $\tilde b_j\in\pico(C)$ so that $3\tilde b_1 = b_3-b_1$, $3\tilde b_2 = b_1 -b_2$, $3\tilde b_3 = b_2 - b_3$.  Depending on whether $\pico(C)\cong \C$ or $\pico(C)\cong\C^*$, these $\tilde b_j$ might or might not be unique, but in either case, they can be chosen so that $\sum \tilde b_j = 0$.  In this case, Proposition \ref{fixingauts} gives us $T\in\Aut(\cp^2)$ fixing $C$ component-wise with multiplier $1$ and translations $\tilde b_j$.  One checks directly that $T\circ f \circ T^{-1}$ has multiplier $1$ and translation $b=b_V$ satisfying
$3b = \sum b_j$ independent of $V\subset C$.  The case when $C$ has two irreducible components can be verified similarly.

From Theorem \ref{quadauts}, we obtain that $p_j^- \sim p_j^+ - 2b$ for each $p_j^-$ in $I(f)$.  Hence $f^2(p_j^-) \sim p_j^+$.  In fact, if $V\subset C$ is the component containing $p_j^+$, then it follows that $p_j^- \in f(V)$ when $C$ has three irreducible components and $p_j^- \in V$ when $C$ has two components.  In any case, we find that $f^2(p_j^-) \in V$, so that $f^2(p_j^-) = p_j^+$.  Since $3b\not\sim 0$, it follows that 
$f(p_j^-)\neq p_j^+$.  If $p_j^- = p_j^+$ for some $j$, then in fact $2b \sim 0$ and $p_j^- = p_j^+$ for all $j$.  Hence $f$ is conjugate to the `standard' quadratic transformation $q$, and the theorem is trivial.  Henceforth,
we assume $p_j^-\neq p_j^+$.

Suppose further for the moment that there are no pairs of indices $j\neq k$ such that $p_j^- = p_k^+$ or $f(p_j^-) = p_k^+$.  Then we may blow up the points $p_j^-, f(p_j^-), p_j^+$ for each $j$ to obtain a rational surface $X$ to which $f$ lifts as an automorphism.  Furthermore, $\sum p_j^- + \sum (p_j^- + b) + \sum p_j^+ \sim -3b + 0 + 3b = 0$.  Finally, one finds by comparing degrees that regardless of the number of components $V\subset C$, each $V$ contains precisely $3\deg V$ of the points blown up.  Hence there is a pencil of cubic curves that contains $C$ and whose basepoints are precisely the ones blown up.  Each curve $C'$ in the pencil intersects each exceptional curve for $f$ precisely once and contains each point in $I(f)$ with multiplicity one.  Comparing degrees, we see that $f(C')$ is another cubic curve containing all the basepoints. We conclude that the pencil lifts to an invariant elliptic fibration of $X$.

Now if it happens that $p_j^- = p_k^+$ or $f(p_j^-) = p_k^+$ for one or more pairs of indices $j\neq k$, then we can reach the same conclusion as before, except that constructing $X$ will require iterated blowing up, the precise nature of which depends on which special case we are in.  The important thing is that since $2b, 3b\not\sim 0$, one always has to blow up nine evenly distributed points in $C_{reg}$ that sum to zero in $\pico(X)$.
\end{proof}

\begin{prop}
\label{olength1}
Let $P$ be the characteristic polynomial for orbit data $n_1,n_2,n_3,\sigma$.  If $n_j =1$ for some $j=\sigma(j)$ that is fixed by $\sigma$, then all roots of $P$ lie on the unit circle\footnote{Since $P$ is monic with integer coefficients, a theorem of Kronecker tells us that all roots are roots of unity.}. 
\end{prop}

\begin{proof}
Suppose e.g. that $j=1$ and that $P$ has a root $\lambda$ with magnitude different from $1$.   Recalling the discussion after Proposition \ref{fstar}, we let $\hat f^*:V\to V$ be the `abstract isometry' associated to the data $1,n_2,n_3,\sigma$.  Then $f^*v = \lambda v$ for some $v\in V$.  

Using the fact that $f_*$ is both inverse and adjoint to $f^*$, we find 
$$
\pair{v}{v} = \pair{v}{\hat f_*\hat f^*v} = \pair{\hat f^*v}{\hat f^*v} = |\lambda|^2 \pair{v}{v}.
$$  
Thus $\pair{v}{v}=0$.  Now it follows from Proposition \ref{fstar} that $\hat f_*(H-E_{1,0}) = H-E_{1,0}$.  Thus
$$
\pair{H-E_{1,0}}{v} = \pair{\hat f_*(H-E_{1,0})}{v} = \pair{H-E_{1,0}}{\hat f^*v} = \lambda\pair{H-E_{1,0}}{v}.
$$
We infer that $\pair{H-E_{1,0}}{v} = 0$.  Since $H-E_{1,0}$ also has vanishing self-intersection, and the intersection form has exactly one positive eigenvalue, it follows that $v$ is a multiple of $H-E_{1,0}$.  Hence $\lambda= 1$ contrary to assumption.
\end{proof}

\section{Irreducible cubics}

\begin{cor}
\label{nodal}
Suppose that $f$ is a quadratic transformation properly fixing a nodal irreducible cubic curve $C$.  If $f$ lifts to an automorphism on some modification $X\to \cp^2$, then the topological entropy of $f$ vanishes.
\end{cor}

\begin{proof}
Since $\pico(C) \cong \C^*$, the multiplier of $f|_{C_{reg}}$ is $\pm 1$.  Since $C$ is irreducible, the assertion follows from Corollary \ref{minus1} and Theorem \ref{plus1}.
\end{proof}

\begin{cor}
\label{smooth generic}
Suppose that $f$ is a quadratic transformation properly fixing a smooth cubic curve $C$.  If $f$ has positive entropy and lifts to an automorphism of some modification $X\to \cp^2$, then either
\begin{itemize}
 \item $C \cong \C/(\Z+i\Z)$ and the multiplier for $f|_C$ is $\pm i$; or
 \item $C \cong \C/(\Z+e^{\pi i/3}\Z)$ and the multiplier for $f|_C$ is a primitive cube root of $-1$.
\end{itemize}
\end{cor}

\begin{proof}
If we are not in one of the two cases described in the conclusion, then the multiplier for $f|_C$ must be a square or cube root of $1$.  From Corollary \ref{minus1} and Theorem \ref{plus1}, we deduce that if $f$ lifts to an automorphism, then the entropy of $f$ is zero. 
\end{proof}

\begin{eg}
\label{smoothsymmetric}
Suppose $C\cong\C/(\Z+i\Z)$.  Then remarkably, there are quadratic transformations properly fixing $C$ and lifting to automorphisms with positive entropy.  For example, Theorem \ref{quadauts} gives us a quadratic transformation $f$ properly fixing $C$ with $I(f) = \{p_1^+,p_2^+,p_3^+\} = \{i/9,4i/9,7i/9\}$ and such that $f|_C$ is given by $z\mapsto iz + 5/9$. Condition \ref{4th} from the same theorem tells us $p_1^- = ip_1^+ - 2b = 7/9$, and similarly $p_2^- = 4/9$, $p_3^- = 1/9$.

Iterating $f$ gives 
$$
p_1^- = 7/9\mapsto 7i/9 + 5/9 \mapsto -7/9 + 5i/9 \mapsto 7i/9 = p_3^+.
$$
Similarly $f^3(p_2^-) = p_1^+$ and $f^3(p_3^-) = p_2^+$.  In summary, $f$ realizes the orbit data $\sigma:1\mapsto 3\mapsto 2$, $n_1=n_2=n_3 = 4$.

On blowing up the twelve points $f^k(p_j^-)$, $0\leq k \leq 3$, $1\leq j\leq 3$, we obtain an automorphism $\hat f:X\to X$.  By \eqref{charpoly}, the characteristic polynomial for $\hat f^*$ is 
$P(\lambda) = \lambda^{13}-2\lambda^{12} + 3\lambda^9-3\lambda^8 + 3\lambda^5 - 3\lambda^4 + 2\lambda - 1$, which has largest root $\lambda_1 = 1.722\dots$.  Hence $\hat f$ has entropy $\log\lambda > 0$.

We make two further observations about this example.  The restriction of $\hat f:X\to X$ to (the proper transform of) $C$ is periodic with period $4$.  Hence $\hat f^4$ is an example of a positive entropy automorphism of a rational surface that fixes a smooth elliptic curve pointwise.  Secondly, since $C$ has negative self-intersection $C^2 = 9-12$ in $X$ and since $f(C) = C$, one can contract $C$ equivariantly to obtain an automorphism $\check f:\check X\self$ with positive entropy on a normal (possibly not projective) surface with a simple elliptic singularity.
\end{eg}

On the other hand, as Eric Riedl points out, not all orbit data that looks plausible (i.e. $n_j\leq 4$) for the `square' torus is actually realizable.

\begin{eg}
Let $C = \C/(\Z+i\Z)$ again, and consider the orbit data $n_1=n_2=n_3=4$, $\sigma=\id$.  If $f$ properly fixes $C$ and realizes this data, then we have $f|_C:z\mapsto iz + b$ for some $b\in C$, and $(f|_C)^3(p_j^-) = p_j^+$.  Since $(f|_C)^4 = \id$, this is equivalent to $f|_C(p_j^+) = p_j^-$.  Hence condition \ref{4th} from Theorem \ref{quadauts} implies $ap_j^+ + b \sim p_j^- \sim ap_j^+ - 2b$, which gives $3b=0$, contrary to the last assertion in the proposition.
\end{eg}

The final irreducible case occurs when $C$ has a cusp, and in this one it is much easier to construct automorphisms.
In order to state our result, let us make a convenient definition.  Suppose we are given orbit data $n_1,n_2,n_3\geq 1$ and a quadratic transformation $f$ properly fixing $C$.  We will say that $f$ \emph{tentatively realizes} the orbit data if $(f|_{C_{reg}})^{n_j-1}(p_j^-) = p_{\sigma_j}^+$ for each $n_j$.  We stress that this does not mean that $f$ \emph{realizes} the orbit data in the fashion described in \S \ref{auts}.  For instance, one might find that $f^{n-1}(p_1^-) = p_{\sigma_1^+}$ for some $n<n_1$ so that $f$ actually realizes the orbit data $n,n_2,n_3,\sigma$ instead of $n_1,n_2,n_3,\sigma$.  Tentative realization is, however, a necessary precondition for realization.

\begin{thm} 
\label{tentative}
Let $C$ be a cuspidal cubic curve, $n_1,n_2,n_3\geq 1$ and $\sigma\in\Sigma_3$ be orbit data.  If $f$ is a quadratic transformation properly fixing $C$ that tentatively realizes this orbit data, then the multiplier for $f|_{C_{reg}}$ is a root of the corresponding characteristic polynomial $P(\lambda)$.  Conversely, there exists a tentative realization $f$ for each root $\lambda = a$ of $P(\lambda)$ that is not a root of unity, and $f$ is unique up to conjugacy by linear transformations preserving $C$.
\end{thm}

\begin{proof}
Since $a\neq 1$ by hypothesis, the restriction $f|_{C_{reg}}$ is given by $f(p)\sim ap + b$ which has a unique fixed point $p_0 \sim b/(1-a)$.  We let $\tilde p = \kappa(p)-\kappa(p_0) \in \pico(C)\cong\C$ for any point $p\in C_{reg}$.  Hence $\widetilde{f^k(p)} = a^k\tilde p$.  Proposition \ref{fixingauts} and the fact that all $a\in\C^*$ are possible multipliers for $C$ allow us to conjugate by $T\in\Aut(\cp^2)$ to arrange that $p_0 \sim \frac13$.  Items \ref{2nd} and \ref{4th} in Theorem \ref{quadauts} then become
\begin{itemize}
\item $\sum \tilde p_j^- = a-2$;
\item $\tilde p_j^- = a\tilde p_j^+ + a-1$, for $j=1,2,3$.
\end{itemize}
Therefore if the points $p_j^-\in C_{reg}$ satisfy the first of these conditions, Theorem \ref{quadauts} gives us a quadratic transformation $f$ that properly fixes $C$ with multiplier $a$ and $I(f^{-1}) = \{p_1^-,p_2^-,p_3^-\}$.  The second condition is just a restatement of condition \ref{4th} in Theorem \ref{quadauts}.

Now $f$ tentatively realizes the given orbit data if and only if $a^{n_j-1}\tilde p_j^- = \tilde p_{\sigma_j}^+$ for 
$j=1,2,3$.  If $\sigma$ is the identity permutation, then in light of the second condition above, this is equivalent to
\begin{equation}
\label{id}
\tilde p_j^- = \frac{a-1}{1-a^{n_j}}, \quad j=1,2,3.
\end{equation}
The first condition in turn gives
$
\sum_j \frac{1}{1-a^{n_j}} = \frac{a-2}{a-1}.
$
One verifies readily that this is equivalent\footnote{This fortunate coincidence is largely accounted for in \S7 of \cite{mcm} whose arguments show that the multiplier $a$ for a tentative realization must be a root of $P(\lambda)$ and conversely that each root of $P(\lambda)$, disregarding multiplicity, gives rise to at least one tentative realization.} to $P(a) = 0$, where $P$ is the characteristic polynomial for the orbit data $n_1,n_2,n_3,\id$.  This proves the theorem when $\sigma = \id$.

The cases where $\sigma$ is an involution or $\sigma$ is cyclic are similar.  If $\sigma$ is the involution swapping e.g. indices $1$ and $2$, then one finds that
\begin{equation}
\label{swap}
\tilde p_1^- = \frac{(a-1)(1+a^{n_2})}{1-a^{n_1+n_2}}, \quad
\tilde p_2^- = \frac{(a-1)(1+a^{n_1})}{1-a^{n_1+n_2}},\quad
\tilde p_3^- = \frac{a-1}{1-a^{n_3}},
\end{equation} 
where $a$ is a root of the characteristic polynomial associated to $n_1,n_2,n_3,\sigma$.
And if $\sigma$ is the cyclic permutation $\sigma:1\mapsto 2\mapsto 3$, then
\begin{equation}
\label{cycle}
\begin{array}{c}
\tilde p_1^- = \frac{(a-1)(1+a^{n_3} + a^{n_2+n_3})}{1-a^{n_1+n_2+n_3}},\quad
\tilde p_2^- = \frac{(a-1)(1+a^{n_1} + a^{n_3+n_1})}{1-a^{n_1+n_2+n_3}},\\ \\
\tilde p_3^- = \frac{(a-1)(1+a^{n_2} + a^{n_1+n_2})}{1-a^{n_1+n_2+n_3}}.
\end{array}
\end{equation}
\end{proof}

As it turns out, most of the tentative realizations given by Theorem \ref{tentative} actually do realize the given orbit data.  

\begin{thm}
\label{actual}
Suppose in Theorem \ref{tentative} that $a$ is a root of $P(\lambda)$ that is not equal to a root of unity, and let $f$ be the tentative realization corresponding to $a$ of the given orbit data $n_1,n_2,n_3,\sigma$.  Then $f$ realizes the orbit data if 
and only if we are not in one of the following two cases
\begin{itemize}
\item $\sigma \neq \id$ and $n_1=n_2=n_3$;
\item $\sigma$ is an involution swapping indices $i$ and $j$ such that $n_i = n_j$.
\end{itemize}
\end{thm}

\begin{proof}
The tentative realization $f$ will necessarily realize \emph{some} orbit data.  The problem occurs when the orbit of some point $p_j^-$ intersects $I(f)$ too soon and/or at the wrong point so that the orbit data that is realized differs from the given data.

That is, we have $f^{n-1}(p_j^-) = p_{\sigma_i}^+$ for some $i,j$ and some positive $n\in\N$, where $i\neq j$ and/or $n<n_i$.  Using the notation from the proof of Theorem \ref{tentative}, this becomes 
\begin{equation}
\label{badrelation}
 a^n \tilde p_j^- = \tilde p_{\sigma_i}^+ = a^{n_i} \tilde p_i^-
\end{equation}
In particular, we may suppose $i\neq j$ since $a$ is not a root of unity.  Since $\tilde p_i^-$, $\tilde p_j^-$ are given by rational expressions (over $\Z$) in $a$, \eqref{badrelation} amounts to a polynomial equation satisfied by $a$.  But $a$ is a root of the characteristic polynomial for the orbit data and therefore by Proposition \ref{fstar} Galois conjugate to $a^{-1}$.  Hence \eqref{badrelation} remains true if we replace $a$ by $a^{-1}$ throughout.  

Assume for now that $\sigma = \id$ or that $\sigma$ exchanges two indices.  Replacing $a$ by $a^{-1}$ in the formula for $\tilde p_j^-$ amounts to replacing $\tilde p_j^-$ by $\tilde p_{\sigma_j}^+ = a^{n_j -1} \tilde p_j^-$.  One can verify this directly using the formulas \eqref{id}, \eqref{swap}.  However, this follows also on general principle from the fact that (given the normalization $p_0 \sim 1/3$) there is a unique tentative realization $g$ of the orbit data $n_1,n_2,n_3,\sigma$ corresponding to the multiplier $a^{-1}$.  Since $\sigma = \sigma^{-1}$, one can relabel indices $j \mapsto \sigma(j)$ and see that $f^{-1}$ gives such a realization. Hence $g=f^{-1}$. The upshot is that $a$ must satisfy the second equation
$a^{-n+n_j} \tilde p_j^- = \tilde p_{\sigma_i}^-$.  Combined with \eqref{badrelation} this implies that $a^{n_i+n_j - 2n} = 1$.  Since by hypothesis $a$ is not a root of unity, it follows that $n_i+n_j=2n$.  

Suppose $n_i\neq n_j$, e.g. $n_i<n_j$.  Then we may write $n_i = n - k$, $n_j = n + k$ for some $k>0$.  Thus the orbit of $p_j^-$ contains that of $p_i^-$ as follows:
$$
p_j^-,\dots,f^k(p_j^-)=p_i^-,\dots,f^{n_j-k-1}(p_j^-) = p_{\sigma_i}^+,\dots,f^{n_j-1}(p_{\sigma_j}^+).
$$
Hence in the blowing up procedure used to lift the birational map $f:\cp^2\to \cp^2$ to an automorphism $\hat f:X\to X$, the orbit segment $p_i^-,\dots, p_{\sigma_i}^-$ is blown up \emph{before} the segment $p_j^-,\dots,p_{\sigma_j}^+$.  Hence despite the coincidence \eqref{badrelation}, $f$ still realizes the given orbit data.  

If instead $n_i = n_j = n$, then \eqref{badrelation} implies $p_i^- = p_j^-$.  Without loss of generality, we may assume that $p_i^-$ is infinite near to $p_j^-$.  Then symmetry of $f$ and $f^{-1}$ implies that $p_i^+$ is infinitely near to $p_j^+$, whereas $n_i=n_j$ implies that $p_{\sigma_i}^+$ is infinitely near to $p_{\sigma_j}^+$.  Hence, under our assumption that $\sigma$ is the identity or a transposition, $f$ realizes the given orbit data if and only if $\sigma = \id$.

Turning to the remaining case, where $\sigma:1\mapsto 2\mapsto 3$ is cyclic, we begin again with \eqref{badrelation}.  Without loss of generality, we further suppose that $j=1$, $i=2$.  Then \eqref{cycle} and \eqref{badrelation} give us $a^n(1+a^{n_3}+a^{n_2+n_3}) = a^{n_2}(1+a^{n_1} + a^{n_1+n_3})$.  Replacing $a$ with $a^{-1}$ in this equation also gives $a^{n_1}(1+a^{n_2}+a^{n_2+n_3}) = a^n(1+a^{n_3} + a^{n_1+n_3})$.  Adding the two equations and simplifying, we obtain that $(a^{n + n_3}-1)(a^{n_1} - a^{n_2}) = 0$.  Since $a$ is not a root of unity, we infer that either $n = -n_3$ or $n_1 = n_2$.  

In the first case, we substitute for $a^{-n_3}$ for $a^n$ in \eqref{badrelation} and find that $(a^{n_3}+1)(a^{n_1+n_2+n_3} - 1) = 0$, which is impossible because $a$ is not a root of unity and $n_1,n_2,n_3 \geq 1$.  In the second case, when $n_1 = n_2$, we rewrite \eqref{badrelation} as
$$
a^n \tilde p_1^- = a\tilde p_{\sigma_2}^+ = a \tilde p_3^+ = \tilde p_3^- + 1-a.
$$
Substituting our formulas \eqref{cycle} for $\tilde p_1^-$ and $\tilde p_3^-$, we obtain that $a^n(1+a^{n_3} + a^{n_3+n_2}) = a^{n_2}(1+a^{n_3} + a^{n_1+n_3})$.  Using $n_1=n_2$, we obtain that either $1+a^{n_3} + a^{n_1+n_3} = 0$ or (since $a$ is not a root of unity) $n = n_2$.  In the first case, we replace $a$ with $a^{-1}$ and deduce finally that $n_1 = n_3$.  In the second case, we return to \eqref{badrelation} and find that $\tilde p_1^- = \tilde p_2^-$, which again gives $n_1 = n_3$.  Regardless, we arrive at the condition $n_1=n_2=n_3$.  From here we obtain a contradiction following the same logic used to rule out the possibility that $n_i=n_j$ when $\sigma$ transposes $i$ and $j$.
\end{proof}

\section{Reducible cubics}

We now deal briefly with the cases where the cubic curve $C$ is reducible with only one singularity---i.e. $C$ consists of three distinct lines through a single point, or $C$ consists of a smooth conic and one of its tangent lines.  In either case, the components of $C_{reg}$ are copies of $\C$, and the story is much the same as it is for cuspidal cubics.  The only additional complication is that a quadratic transformation cannot realize given orbit data unless the permutation it induces on the components of $C$ is compatible with the permutation $\sigma$ in the orbit data.

\begin{thm}
\label{concurrent lines}
Let $C$ be the plane cubic consisting of three lines meeting at a single point.  Let $n_1,n_2,n_3\in\N$, $\sigma\in\Sigma_3$ be orbit data whose characteristic polynomial $P(\lambda)$ has a root outside the unit circle. Then the orbit data can be realized by a quadratic transformation $f$ that properly fixes $C$ if and only if one of the following is true:
\begin{itemize}
 \item $\sigma=\id$;
 \item $\sigma$ is cyclic and either all $n_j\equiv 0 \operatorname{mod} 3$ or all 
       $n_j\equiv 2\operatorname{mod} 3$;
 \item $\sigma$ is a transposition (say $\sigma$ interchanges $1$ and $2$) and either $n_1$ and $n_2$ are odd, or
       no two $n_j$ are the same $\operatorname{mod} 3$ and $n_3\equiv 0\operatorname{mod} 3$.
\end{itemize}
If one of these holds, we can arrange that $f|_{C_{reg}}$ has multiplier $a$ where $a$ is any root of $P$ that
is not a root of unity.  The choice of $a$ determines $f$ uniquely up to linear conjugacy.
\end{thm}

\begin{proof}
We only sketch the argument.  Let $V_j\subset C_{reg}$ denote the component containing $p_j^+$. Since $a\neq 1$, the restriction $f|_{V_j}$ has a unique `fixed point' $p_j\sim f(p_j)$.  Using Proposition \ref{fixingauts} we may conjugate by an element of $\Aut(\cp^2)$ to arrange that $z_j = \frac{1}{3(a-1)}$ for all $j=1,2,3$.  
Hence $f(p)\sim a(p-p_j) + p_j$ has the same expression on each $V_j$.

Given orbit data whose characteristic polynomial $P$ has a root $a$ that is not a root of unity, we can repeat the arguments used to prove Theorem \ref{tentative} to prove that there exists a quadratic transformation $f$ properly fixing $C$ such that the multiplier of $f|_{C_{reg}}$ is $a$ and $f^{n_j-1}(p_j^-)\sim p_j^+$ for each $j=1,2,3$.  Indeed given $a$ and the fixed points $p_j$, $f$ is determined up to permutation of the $V_j$.  Let us write $f(V_j) = V_{s_j}$ where $s\in\Sigma_3$.

Now each $V_j$ contains one point of indeterminacy---say $p_j^+\in V_j$; and $p_j^-$ therefore lies in $f(V_j) = V_{s_j}$.  Therefore if $\sigma=\id$, we also choose $s=\id$, and then $f^{n_j-1}(p_j^-) \sim p_j^+$
implies $f^{n_j-1}(p_j^-) = p_j^+$.  Hence $f$ realizes the given orbit data.

If $\sigma$ is cyclic (say $\sigma:1\mapsto 2\mapsto 3$), then certainly $f$ must permute the $V_j$ transitively.  That is, $s$ must also be cyclic.  If $s = \sigma$, then we have $p_j^- \in V_{\sigma_j}$.  Hence
$f^{n_j-1}(p_j^-)$ lies in $V_j$ if and only if $n\equiv 0\mod 3$.  That is, when $s=\sigma$ then $f$ realizes the given orbit data if and only if each $n_j\equiv 0\mod 3$.  To realize orbit data for which $n_j\equiv 2\mod 3$, 
one may check that it is similarly necessary and sufficient that $s=\sigma^{-1}$.  We note that the exceptional cases from Theorem \ref{actual} need not concern us here, because different points of indeterminacy lie in different components of $C_{reg}$ and cannot therefore coincide.

The case where $\sigma$ is a transposition can be analyzed similarly.  The case where $n_1$ and $n_2$ are odd can be realized by a quadratic transformation $f$ that swaps $V_1$ and $V_2$ while fixing $V_3$.  The other case can be achieved by letting $f$ permute the $V_j$ cyclically.
\end{proof}

When $C$ is the union of a smooth conic with one of its tangent lines, one has a result similar to Theorem \ref{concurrent lines}.  However, in this situation it will always be the case that the conic portion of $C$ contains more than one point of indeterminacy.  Since such points of indeterminacy might coincide, it is necessary to hypothesize away exceptional cases like those in Theorem \ref{actual}.  The upshot is that the analogue of Theorem \ref{concurrent lines} for $C$ equal to a conic and a tangent line is somewhat messy to state.  Since it is not conceptually different, we omit it.

\subsection{Reducible cubics with nodal singularities.}

Finally, we consider reducible cubics with more than one singularity.  As above, we devote more attention to the case of a cubic with three irreducible components.

\begin{thm}
\label{3linenogos}
Suppose $f:\cp^2\to\cp^2$ is a quadratic transformation that properly fixes $C=\{xyz=0\}$ and lifts to an automorphism with positive entropy on some blowup of $\cp^2$.  Then $f$ fixes $C_{reg}$ component-wise and $f|_{C_{reg}}$ has multiplier $1$.  Hence $f$ realizes orbit data of the form $n_1,n_2,n_3\geq 1$, $\sigma = \id$.
\end{thm}

\begin{proof}
Since $\pico(C) \cong \C^*$, the multiplier of $f|_{C_{reg}}$ is $\pm 1$.  We claim that the multiplier of $f$ is $-1$ if and only if $f$ swaps two components of $C_{reg}$ and preserves the other.  Indeed, if $f$ fixes $\{z=0\}$ while swapping $\{x=0\}$ and $\{y=0\}$, then in particular, $f$ interchanges the points $[0,1,0]$ and $[1,0,0]$.  Hence the multiplier of $f|_{C_{reg}}$, which is the same as that of $f|_{\{z=0\}}$, is $-1$.  Similarly, if $f$ fixes all three components of $C_{reg}$, then it also fixes all three singularities of $C$, and we infer that $f$ has multiplier $+1$.  Finally, if $f$ cycles the components of $C_{reg}$, then $f^3$ fixes $C_{reg}$ component-wise, and we infer again that the multiplier of $f|_{C_{reg}}$, which is the same as that of $f^3|_{C_{reg}}$, is $+1$.  This proves our claim.

Suppose now that the multiplier is $-1$ and, without loss of generality, that $f$ fixes the component $V\subset C_{reg}$ containing $p_1^\pm$. Hence $f^2|_V = \id$ and $\sigma_1 = 1$.  It follows that $n_1 = 1$ or $n_1 = 2$.  If $n_1=2$, then on the one hand, we have $p_1^+ \sim -p_1^+ + b_1$, where $b_1$ is the translation for $f|_V$.  And on the other hand, we have from Theorem \ref{quadauts} that $p_1^- \sim - p_1^+ - b_2-b_3$ where $b_2,b_3\in\C^*$ are the translations on the other two components of $C_{reg}$.  We infer that $b_1+b_2+b_3 = 0$, and by item \ref{3rd} in Theorem \ref{quadauts} that $\sum p_j^+ \sim 0$.  This contradicts the fact that the points in $I(f)$ cannot be collinear.  So $n_1 = 1$.  From Proposition \ref{olength1}, it follows that the automorphism induced by $f$ has entropy $0$, contrary to hypothesis.

Hence the multiplier for $f|_{C_{reg}}$ is $+1$.  If $f$ permutes the components of $C_{reg}$ cyclically, then Proposition \ref{toofew} and Theorem \ref{plus1} imply that $f$ lifts to an automorphism with zero entropy, again counter to our hypothesis. We conclude that $f$ fixes $C$ component-wise.
\end{proof}

Having just ruled out many types of orbit data on $C=\{xyz=0\}$, we consider whether the remaining cases may be realized.  Let $n_1,n_2,n_3\geq 1$, $\sigma = \id$ be orbit data and $f$ be a quadratic transformation that fixes $C$ component-wise with multiplier $1$.  Then we have $f(p)\sim p + b_j$, on the component containing $p_j^\pm$.  Theorem \ref{quadauts} gives $p_j^- \sim p_j^+ + b_j - b$ where $b = b_1 + b_2 + b_3$; and $f$ tentatively realizes the given orbit data if $p_j^+ \sim p_j^- + (n_j -1)b_j$.  We infer $n_j b_j = b$ for $j=1,2,3$.

Note that these equations hold relative to the group structure on $\pico(C) \cong (\C/\Z,+)$.  For convenience we will confuse equivalence classes with their representatives here, regarding $b$, $b_j$ as elements of $\C$ instead of $\C/\Z$.  The previous equations must then be we must understood `mod 1': e.g. $n_j b_j = b+m_j$ for some $m_j\in\Z$.  Solving for $b_j$ and summing over $j$ gives 
$$
b\left(1-\sum\frac{1}{n_j}\right) = \sum \frac{m_j}{n_j},
$$
which implies 
\begin{equation}
\label{mtob}
b_j = \frac{m_j}{n_j} + \frac{1}{n_j}\frac{m_1n_2n_3 + m_2 n_3 n_1 + m_3 n_1 n_2}{n_1n_2n_3 - n_1n_2 - n_2n_3-n_3n_1}.
\end{equation}
On the other hand, it is not difficult to see from Theorem \ref{quadauts} that if $m_1,m_2,m_3\in\Z$ is any choice of integers, then we get a tentative realization of our orbit data.

\begin{prop} 
\label{tentative3}
Let $C=\{xyz=0\}$ and $n_1,n_2,n_3,\sigma = \id$ be orbit data.  Then this data may be tentatively realized by a quadratic transformation $f$ properly fixing $C$ if and only if $n_1n_2n_3 \neq n_1 n_2 + n_2 n_3 + n_3 n_1$.  Any such $f$ has translations $b_j$, $j=1,2,3$ given by equation \eqref{mtob}.  Conversely, any choice of $m_1,m_2,m_3\in\Z$ in \eqref{mtob} determines a tentative realization $f$ that is unique up to linear conjugacy.
\end{prop}

\begin{proof}
The above discussion shows that the restrictions on $f$ are necessary and sufficient for $f$ to tentatively realize the orbit data.  We need only argue that there actually exists a quadratic transformation $f$ that satisfies the restrictions.  For this we rely on the existence portion of Theorem \ref{quadauts}.  Note that the above discussion also shows that while the conditions $f^{n_j-1}(p_j^-) = p_j^+$ constrain the translations $b_j$, they do not (otherwise) constrain the points $p_j^\pm$.  Hence we need only adhere to the conditions \ref{1st} and \ref{3rd} in Theorem \ref{quadauts}, choosing $p_j^+$ so that
$\sum p_j^+ \sim b$ and then $p_j^- \sim p_j^+ + b_j - b$.  From Proposition \ref{fixingauts}, we see in fact that we can always conjugate by a linear transformation to obtain $p_1^+ \sim p_2^+ \sim 0$ and $p_3^+ \sim b$.
\end{proof}

Since the points of indeterminacy for $f$ lie in different components of $C$, the only way the transformations $f$ in the proposition can fail to realize the given orbit data is if $f^k(p_j^-) = p_j^+$ for some $0\leq k\leq n_j-2$.  This happens if and only if $f^\ell p_j^- = p_j^-$, i.e.  $\ell b_j \in\Z$, for some $0< \ell <n_j-2$.

\begin{thm}
\label{actual3}
Let $C=\{xyz=0\}$ and consider orbit data of the form $n_1\geq n_2 \geq n_3\geq 2$, $\sigma = \id$ for which the corresponding characteristic polynomial has a root outside the unit circle.  Then there exists a quadratic transformation properly fixing $C$ and realizing this orbit data if and only if we are not in one of the following cases.
\begin{itemize}
 \item $n_2+n_3 \leq 6$; 
 \item $n_3 =2$, and $n_1=n_2= 5$ or $n_1=n_2=6$;
 \item $n_1=n_2=n_3 = 4$.
\end{itemize}
\end{thm}

\begin{proof}
If a quadratic transformation $f$ realizes orbit data $n_1\geq n_2 \geq n_3$, then it must be one of the tentative realizations from Proposition \ref{tentative3}.  By Proposition \ref{olength1} we may assume $n_3\geq 2$.  If $n_2=n_3 = 2$, we have $n_1n_2n_3 - n_1 n_2 - n_2 n_3 - n_3 n_1 = 0$ contrary, so by Proposition \ref{tentative3}, we may assume $n_2\geq 3$.

Now if $n_2 = 3, n_3 = 2$, equation \eqref{mtob} gives 
$$
b_1 = \frac{m_1 + 2m_2 + 3m_3}{n_1-6}
$$
Hence $\ell b_1 \in \Z$ for $\ell = n_1 - 6 \leq n_1 - 2$.  That is, every tentative realization of the orbit data
$n_1, 3,2,\id$ fails to actually realize this data.  The same argument rules out orbit data with $n_2 = 4, n_2 =2$ or $n_2 = n_3 = 3$.

We are left with three remaining bad cases.  The data $n_1=n_2 = 5$ and $n_3 =2$ is ruled out in the same way as the previous cases.  Suppose $n_1=n_2=n_3 = 4$.  This time \eqref{mtob} tells us that for any tentative realization, the translations are given by 
$$
b_j = \frac{m_j + (m_1+m_2+m_3)}{4},
$$
where $m_1,m_2,m_3\in\Z$.  Thus the numerator will be even for some $j$, which implies $(n_j-2)b_j = 2b_j \in\Z$. Hence the data is not realized.  Similar arguments rule out the data $n_1=n_2=6$, $n_3=2$.

Turning to the good cases, we first assume $n_2>n_1\geq 4$.  We set $m_1=1$, $m_2=m_3=0$ and take $f$ to be the tentative realization from Proposition \ref{tentative3}.  Then \eqref{mtob} gives
$$
0 < b_1 = \frac{n_2 n_3 - n_2 - n_3}{n_1 (n_2n_3 - n_2 + n_3 - n_2 n_3)-n_2 n_3} 
    = \frac{1}{n_1 - \frac1{1-n_2^{-1}-n_3^{-1}}} < \frac{1}{n_1-2}.
$$
Hence $ 0 < \ell b_1 < 1$ for all $0<\ell \leq n_1-2$.  Similarly, we find for $j=2,3$ that $0<\ell b_j < 1$ for all $0<\ell<n_j-2$.  We conclude that $f$ actually realizes the given orbit data.

The same argument works when $n_1>n_2=n_3 =4$ except that we set $m_2=1$ and $m_1 = m_3 = 0$ in choosing $f$;  it works for $n_2>n_3=3$ if we set $m_1=1$, $m_2=0$, $m_3=-1$; it works for $n_1 > n_2 \geq 5$ and $n_1\neq n_2$ if we set $m_1 = 1$, $m_2 = -1$.

The final case we need to consider is $n_3=2$ and $n_1=n_2\geq 7$.  This time we set $m_1=1$, $m_2=m_3=0$.  It
follows that $0<\ell b_2 < 1$ for all $0<\ell\leq n_2-2$.  It also follows that $b_3\notin\Z$.  For $b_1$, however, things are a bit more delicate.  One shows here that $0<\ell b_1 < 1$ for all $0<\ell\leq n_1-3$ but 
$1<(n_1-3)b_1<2$.  Regardless, the data is realizable.
\end{proof}

Of course, each realization $f$ given by Theorem \ref{actual3} lifts to an automorphism $\hat f:X \to X$ on the rational surface $X$ obtained by blowing up orbit segments $p_j^-,\dots,f^{n_j-1}(p_j^-)$.  These automorphisms are broadly similar to those in Examples \ref{smoothsymmetric}.  That is, some iterate $\hat f^k$ restricts to the identity on the proper transform $\hat C$ of $C$ in $X$.  And in a different direction, the intersection form is negative definite for divisors supported on $\hat C$, so by Grauert's theorem \cite[page 91]{bhpv} one can collapse $\hat C$ to a point and obtain a normal surface $Y$ with a cusp singularity to which $\hat f$ descends as an automorphism.

The other reducible cubic curve with nodal singularities is the one with two components $C=\{z(xy-z^2)=0\}$.  
As with $\{xyz=0\}$, there are infinitely many sets of orbit data that can be realized by quadratic transformations
fixing $C$ and also infinitely many that cannot be realized.  Rather than give the complete story, we make some broad observations and give examples indicating the range of possibilities.

\begin{thm}
Suppose that $C=\{z(xy-z^2)\}$ is the reducible cubic  with two singularities.  If $f$ is a quadratic transformation
realizing orbit data $n_1,n_2,n_3,\sigma$ whose characteristic polynomial has a root outside the unit circle, then
$f$ fixes $C$ component-wise and $f|_{C_{reg}}$ has multiplier $1$.  Moreover, either
\begin{itemize}
\item $\sigma$ is a transposition; or
\item $\sigma =\id$ and two of the $n_j$ are equal.
\end{itemize}
\end{thm}

\begin{proof}
The possible multipliers for $C$ are $\pm 1$.  Let $b,c\in\C^*$ denote the translations of $f$ on
$\{xy-z^2\}$ and $\{z=0\}$, respectively.

Suppose that the multiplier is $-1$.  Then by Corollary \ref{minus1}, $f$ switches the two components of $C_{reg}$.  Then $f^2(p)\sim p+(b-c)$ on the conic $\{xy-z^2\}$ and $f^2(p)\sim p+(c-b)$ on $\{z=0\}$.  Moreover, degree considerations force all points $p_j^\pm$ of indeterminacy for $f$ and $f^{-1}$ to lie on this conic.  Hence from Theorem \ref{quadauts} we have $p_j^- + p_j^+ \sim b-c$ for $j=1,2,3$; and
$\sum p_j^- \sim \sum p_j^+ \sim -2b -c$.  Combining all the formulas gives
$$
-3(b+c) \sim \sum(p_j^+ + p_j^-) \sim -2b-4c,
$$
which implies that $b-c = 0$.  Hence $f^2 =\id$ on $C$.  It follows that $f$ can only realize orbit data for which
all orbit lengths satisfy $n_j\leq 2$.  Proposition \ref{toofew} now implies that all roots of the characteristic polynomial have magnitude $1$, contrary to hypothesis.

We can assume therefore that $f|_{C_{reg}}$ has multiplier $+1$.  Theorem \ref{plus1} implies that $f$ fixes $C$ component-wise.  Comparing degrees, we find that $\{xy-z^2\}$ contains two points, say $p_1^+,p_2^+$, of $I(f)$ and
$\{z=0\}$ contains $p_3^+$.  Since the components map to themselves, it follows that $p_1^-,p_2^- \in \{xy=z^2\}$ and $p_3^- \in \{z=0\}$.   Theorem \ref{quadauts} gives
$$
p_1^- - p_1^+ \sim p_2^- - p_2^+ \sim - b - c, \quad p_3^+-p_3^- \sim -2b, \quad \sum p_j^- \sim -2b-c.
$$
The permutation $\sigma$ in the orbit data must fix the index $3$.  Hence either $\sigma = \id$ or $\sigma$ switches the indices $1$ and $2$.  Suppose we are in the former case.  Then for $j=1,2$, we have $p_j^+ - p_j^- \sim (n_j-1)b$.  Combining this with the formulas above gives $(n_j-1) b \sim c$ and hence $(n_2-n_1)b \sim 0$.  So if $n_2 \neq n_1$, we see that $b \sim m/n$, where $0<n < \max\{n_1-1,n_2-1\}$ and $0\leq m < n$ are integers.  So if, say, $n_2 \geq n_1$, we find $f^{n_j-n-1}(p_2^-) \sim p_2^+$ and therefore $f$ does not realize the given orbit data.  It follows that $n_2 = n_1$.
\end{proof}

\begin{eg} We can realize the orbit data $n_1=n_2=5$, $n_3=4$, $\sigma = \id$ on $C=\{(xy-z^2)z=0\}$ as follows.
Choose $p_1^-,p_2^-\in\{xy=z^2\}$ so that $p_1^- \sim 0 \in \C/\Z$, $p_2^-\sim i$, and $p_3^-\in \{z=0\}$ so that $p_3^- \sim -i-5/7$.  Then from Theorem \ref{quadauts} we obtain a quadratic transformation $f$ with $I(f^{-1}) = \{p_1^-,p_2^-,p_3^-\}$ that properly fixes each component of $C$, acting on $\{xy=z^2\}$ by $f(p)\sim p + 1/7$ and on $\{z=0\}$ by $f(p)\sim p + 3/7$.  Also, we obtain that the points in $I(f)$ satisfy $p_3^- = p_3^+ - 2/7$, and that for $j=1,2$, $p_j^- \sim p_j^+ + 4/7$.  Since for each $j$, the points $p_j^+$ and $p_j^-$ lie in the same component of $C$, we infer that $f^3(p_3^-) = p_3^+$ and that for $j=1,2$, $f^4(p_j^-) = p_j^+$.  Hence $f$ tentatively realizes the given orbit data.  Since as one verifies directly, all 14 points $p_j^-,\dots,f^{n_j-1}(p_j^-)$, $j=1,2,3$ are distinct, we conclude that $f$ realizes the give orbit data.
\end{eg}

\begin{eg}
Let $p_1^-,p_2^- \in \{xy=z^2\}$ be given by $p_1^- \sim 8/13$, $p_2^- \sim 0$, and $p_3^-\in \{z=0\}$ by $p_3^-\sim 12/13$.  Then from Theorem \ref{quadauts}, we get a unique quadratic transformation $f$ with $I(f^{-1}) = \{p_1^-,p_2^-,p_3^-\}$ that properly fixes each component of $C$, acting by $f(p)\sim p+3/13$ on $\{xy=z^2\}$ and by $f(p)\sim p+1/13$ on $\{z=0\}$.  The points in $I(f)$ are given by $p_1^+ \sim 12/13$, $p_2^+ \sim 4/13$, $p_3^+\sim 5/13$.  From this information, one verifies that $f$ realizes the orbit data $n_1=3$, $n_2 = 4$, $n_3=7$, $\sigma = (12)$.
\end{eg}

\section{Appendix: the group law on a plane cubic curve}
\centerline{\it by Igor Dolgachev}
\vspace*{.2in}

Let $C$ be a reduced connected projective algebraic curve over an algebraically closed field $\bbK$.  Let  $\Pic(C)$ be the group of isomorphism classes of invertible sheaves on $C$. The exact sequence of abelian groups associated with the exact sequence of abelian sheaves 
$$
1\to \calO_C^* \to \calK_C^* \to \calK_C^*/\calO_C^* \to 1
$$
identifies $\Pic(C) \cong H^1(C,\calO_C^*)$ with the group  $\Div(C) = \Gamma(C,\calK_C^*/\calO_C^*)$ of Cartier divisors modulo principal Cartier divisors $\ddiv(f)$, the images of $f\in \Gamma(C,\calK_C^*)$ in $\Div(C)$. Here $\calK_C$ is the sheaf of total rings of fractions of the structure sheaf $\calO_C$ on $C$. We employ the usual notation for linear equivalence of Cartier divisors $D\sim D'$ (note that this is different than the meaning of $\sim$ elsewhere in the paper), letting $[D]$ denote the linear equivalence class of a Cartier divisor $D$. 

For any $D\in \Div(C)$ and any closed point  $x\in C$ a representative $\phi_x$ of the image $D_x$ of $D$ in $\calK_{C,x}^*/\calO_{C,x}^*$ in $\calK_{C,x}^*$ is called a local equation of $D$ at $x$. The homomorphism $\Div(C)\to H^1(C,\calO_C^*)$ assigns to a Cartier divisor $D$ the isomorphism class of the invertible sheaf $\calO_C(D)$ whose sections over an open subset $U$ are elements $f \in \calK_C(U)^*$ such that, for any $x\in C$, we have  $f_x\phi_x\in \calO_{C,x}$. The correspondence $D\mapsto \calO_C(D)$ defines an isomorphism between the group of linear equivalence classes of Cartier divisors and the group of isomorphism classes of invertible sheaves. Each group will be identified with the group $\Pic(C)$.

A Cartier divisor $D$ is called \emph{effective} if all its local equations can be chosen from $\calO_{C.x}$. An effective Cartier divisor can be considered as a closed subscheme of $C$. The number $h^0(\calO_D) = \dim_{\bbK} H^0(C,\calO_D)$ is called the \emph{degree} of $D$ and is denoted by $\deg D$. Every Cartier divisor $D$ can be written uniquely as a difference $D_1-D_2$ of effective divisors (one uses the additive notation for the group of divisors). The degree of $D$ is defined as the differences of the degrees $\deg D = \deg D_1-\deg D_2$. The degree of a principal divisor is equal to  $0$, and this allows one to define $\deg \calL$ for any invertible sheaf of $C$. An equivalent definition (see \cite{Mumford}) is
$$\deg \calL = \chi(C,\calL)-\chi(C,\calO_C).$$
The Riemann-Roch Theorem on $C$ becomes  equivalent to the assertion that 
$$
\deg:\Pic(C) \to \Z, \quad \calL\mapsto \deg \calL,
$$
is a homomorphism of abelian groups. 

A global  section $s:\calO_C\to \calL$ defines, after taking the transpose ${}^ts:\calL^{-1}\to \calO_C$, a closed subscheme of $C$ with the ideal sheaf  ${}^ts(\calL^{-1})$. If its support is finite, then it is an effective Cartier divisor denoted by $\ddiv(s)$. In this case $\calO_C(\ddiv(s)) \cong \calL$.

A Cartier divisor supported in the set $\creg$ of closed nonsingular points of $C$ is called a Weil divisor. It can be identified with an element of the free abelian group generated by the set $\creg$. 

Let  $V_1,\ldots,V_r$ be the irreducible components of $C$.  Denote by $\iota_j:V_j\hookrightarrow C$ the corresponding closed embeddings. For any invertible sheaf $\calL$ on $C$ we denote by $\deg_j\calL$ the degree of $\iota_j^*\calL$. The multi-degree vector $\bf{deg}(\calL) = (\deg_1\calL,\ldots,\deg_r\calL)\in \Z^r$ defines a surjective homomorphism
$\Pic(C) \to \Z^r.$
The kernel of this homomorphism is denoted by $\Pic^0(C)$. 

Next we assume that $C$ is  a connected reduced curve of arithmetic  genus 1 lying on a nonsingular projective surface $X$. Recall that the arithmetic genus $p_a(C)$ is defined to be equal to $\dim_\bbK H^1(C,\calO_C)$. Thus we have $\chi(C,\calO_C) = 0$ and hence 
$\chi(C,\calL) = \deg \calL.$
The Serre Duality Theorem gives $H^1(C,\calL) \cong H^0(C,\calL^{-1}\otimes \omega_C)$, where $\omega_C$ is the canonical sheaf on $C$. By the adjunction formula,   $\omega_C = \omega_X\otimes \calO_X(C)\otimes \calO_C$, where $\omega_X$ is the canonical sheaf on $X$. Since $H^0(C,\omega_C) \cong H^1(C,\calO_C) \cong \bbK$, we obtain that  $\omega_C$ has a nonzero section whose restriction to each component is non-zero.  The zero divisor of this section is an effective divisor of degree 0, hence the trivial divisor. Thus $\omega_C \cong \calO_C$. This easily implies

\begin{lem}\label{L1} Assume that $\deg_j\calL \ge 0$ for any irreducible component $V_j$ of $C$. Then 
$$
\dim H^0(C,\calL) = \deg \calL.
$$
Moreover, each nonzero section has finite support.
\end{lem}

The following lemma describes the structure of a reduced connected curve of arithmetic genus 1. Its proof is standard (see \cite{Reid}, 4.8) and is omitted. 

\begin{lem}\label{L2} Let $C$ be a connected reduced curve of arithmetic genus 1 lying on a nonsingular projective surface $X$. Let $V_1,\ldots,V_r$ be its irreducible components.
\begin{itemize}
\item[(i)] If $r =1$, i.e. $C$ is irreducible, then either $C$ is nonsingular, or has a unique singular point, an ordinary node or an ordinary cusp.
\item [(ii)] If $r > 1$, then each $V_i$ is isomorphic to $\cp^1$ and $V_i\cdot (C-V_i) = 2.$
\end{itemize}
\end{lem}

The structure of $C$ makes convenient to index the components of $C$ by the cyclic group $\Z/r\Z$, so that each component $V_i$ either intersects  $V_{i-1}$ and $V_{i+1}$ transversally at one point, or $r = 2$ and $V_i$ is tangent to $V_{i+1}$, or $r = 3$ and $V_i$ intersects $V_{i-1}$ and $V_{i+1}$ transversally at the same point.

The following lemma is crucial for defining a group law on the set $\creg$.

\begin{lem}\label{L3} Let $\calL\in \Pic(C)$  with $\deg \iota_i(\calL) = 1$ and $\deg \iota_k(\calL) = 0$ for $k\ne i$.  Then 
$$\calL \cong \calO_C(x_i)$$
for a unique nonsingular closed point $x_i$ on  $V_i$.
\end{lem}

\begin{proof} Without loss of generality we may assume that $i = 0$. By Lemma \ref{L1}, we have $\dim H^0(C,\calL) = 1$. Let $s$ be a nonzero section of $\calL$.  Suppose $\iota_j^*(s) \ne  0$ for all $j$. Then $s$ has only finitely many zeros, hence the divisor of zeros $D$ satisfies $\calO_C(D) \cong \calL$. This implies that $\deg D = 1$ and $D$ is a Weil divisor  $1\cdot x_0$ for some nonsingular point $x_0\in V_0$ (we use that for any singular point $y$ and $\phi_y$ from the maximal ideal of $\calO_{C,y}$ we have $\dim \calO_{C,x}/(\phi_y) \ge 2)$.

Now assume that $\iota_j^*(s) =  0$ for some component $V_j$. Then $\iota_{j+1}^*(s)$ and $\iota_{j-1}^*(s)$ vanish at the points $V_j\cap V_{j+1}$ and   $V_j\cap V_{j-1}$. Since a sheaf of degree zero cannot have a non-zero section  vanishing at some point, we see that $\iota_{i}(s) = 0$ for any component $V_i$ intersecting $V_j$ and different from $V_0$. Replacing $j$ with $i$ and continuing in this way, we may assume that $j = 1$. Thus the divisor of zeros of $\iota_0^*(s)$ contains the divisor of degree 2 equal to $V_0\cap (C\setminus V_0)$. Since $\deg \iota_0^*(\calL) = 1$, this is impossible. 
\end{proof}

\begin{cor} Let $V_j$ be an irreducible component of $C$ and $\frako_j$ be a point on  $V_j$. The map 
$$
\kappa_j:V_j\cap \creg \to \Pic^0(C), \ x\mapsto \calO_C(x-\frako), \text{or}\ x\mapsto [x-\frako_j],
$$
is bijective. If used to define a structure of a group on $V_j\cap \creg$, this group becomes isomorphic to the group of points on an elliptic curve (resp. the multiplicative group $\bbK^*$ of $\bbK$, resp. the additive group $\bbK^+$ of $\bbK$) if  $V_j$ is  smooth curve of genus 1 (resp. an  irreducible nodal  curve or $V_j$ intersects $C\setminus V_j$ at two points, resp. an irreducible  cuspidal curve, or $V_j$ intersects $C\setminus V_j$ at one point).
\end{cor}

\begin{proof} It follows from  Lemma \ref{L1} that the map $\kappa_j$ is injective (no two closed points are linearly equivalent on $C$). For any $\calL\in \Pic^0(C)$, the sheaf $\calL\otimes \calO_C(\frako_j)$ has degree 1 on $V_j$ and degree $0$ on other components.  By Lemma \ref{L3}, $\calL\cong \calO_C(\frako_j)$ is isomorphic to $\calO_C(x)$ for a unique point $x\in V_j$. This checks the surjectivity of the map $\kappa_j$. 

The transfer of the group law on $\Pic^0(C)$ defined by the map $\kappa_j$ reads:
$x\oplus y$ is the unique point on $V_j\cap\creg$ such that 
$$x\oplus y \sim x+y-\frako_j.$$
Assume first that $C = V_0$ is irreducible. Let $\nu:Y\to C$ be the normalization map. If $C$ is a nodal curve, then $\nu^{-1} = p_1+p_2$ and we can identify $\calO_C$, via $\nu^*$, with the subsheaf of  $\calO_Y$ of functions $\phi$ such that  $\phi(p_1) = \phi(p_2)$.  Let  $f:Y \to  \cp^1$ be an isomorphism such that $f^{-1}(0) = p_1, f^{-1}(\infty) = p_2$, where we choose projective coordinates $[t_0,t_1]$ on $\cp^1$ and denote $0 = [1,0], \infty = [0,1]$. The rational function $f$ identifies  the fields of rational functions $R(C) = \Gamma(C,\calK_C)$ and $R(Y) = \Gamma(Y,\calK_Y)$ on $C$ and $Y$. Any nonsingular point $x\in C$ is identified with a point $[t_0,t_1]$ on $\cp^1\setminus \{0,\infty\}$. The latter set is identified with $\bbK^*$ by sending $[t_0,t_1]$ to $t = t_1/t_0\in \bbK^*$. Now choose $\frako =\frako_0 =  1$. Then, for any $x,y,z\in \creg$, 
$x+y\sim \frako+z$ if and only if there exists  a rational function $r(t) = (t-x)(t-y)/(t-1)(t-z)$ with  
$r(0) = r(\infty)$. The latter condition implies $xy/z = 1$, hence $z = x\oplus y = xy.$ This defines an isomorphism of groups $\creg \cong \bbK^*$.

Using similar notation if $C$ is a cuspidal curve, we have $\nu^{-1} = 2p$ for some point $p\in Y$. We may identify $\calO_C$ with the subsheaf of $\calO_Y$ of functions $\phi$ such that  $\phi-\phi(p)\in \frakm_{Y,p}^2$. Now we identify 
$\creg$ with $\cp^1\setminus \{\infty\}$ and take $\frako = 0$. Then we have $x+y\sim \frako+z$ if and only if there exists  a rational function $r(t) = (t-x)(t-y)/(t-1)(t-z)$ such that $r-r(\infty) = r(t)-1$ has zero at $\infty$ of order $2$.  It is easy to see that this gives the condition $z= x\oplus y = x+y$. This defines an isomorphism of groups $\creg \cong \bbK^+$.

Now let us assume that $C$ is reducible and $V_i\cdot (C-V_i)$ consists of two points.  We identify $\calO_C$ with the subsheaf of $\prod \calO_{V_i}$ whose sections on an open subset $U$ are those $(\phi_1,\ldots,\phi_r), \phi_i\in \Gamma(U\cap V_i,\calO_{V_i}),$ such that $\phi_i(V_i\cap V_j) = \phi_j(V_i\cap V_j)$. We identify each $V_i$ with $\cp^1$ and  assume that $V_j = V_0$.  If $r > 2$, we identify  the point $V_0\cap V_{1}$ with $0$, and the point $V_0\cap V_{-1}$ with $\infty$. In the case $r = 2$, we set $V_0\cap V_1 = \{0,\infty\}$. Now we choose $\frako_0= 1$. For any $x,y,z\in V_0$, we have 
$x+y\sim \frako+z$ if and only there exists a rational functions $f_i$ such that $f_0 = (t-x)(t-y)/(t-1)(t-z), $ and $f_i $ are constants for $i\ne 0$ such that $r(0) = f_1 = f_2= \ldots f_{-1} = r(\infty)$. This implies that $xy= z$ and shows that $V_0\cap\creg$ is isomorphic to $\bbK^*$. 

We leave the case when $V_j\cap (C\setminus V_j)$ consists of one point to the reader. 
\end{proof}

Now let us define the group law on $\creg$.  We fix some $\frako_j$ on each $V_j\cap\creg$. The group law will depend on this choice. We designate $\frako_0$ to be the zero element.

By Lemma \ref{L3}, for any $x_i\in V_i\cap\creg, x_j\in V_j\cap\creg$,
$$\calO_C(x_i+x_j-\frako_i-\frako_j+\frako_{i+j}) \cong \calO_C(y)$$
for some unique  point $y\in V_{i+j}\cap\creg$. We define the group law by setting
$$x_i\oplus x_j: = y.$$
In other words,  by definition,
$$x_i\oplus x_j \sim x_i+x_j+\frako_{i+j} - \frako_i-\frako_j.$$
It is immediately checked that the binary operation $\oplus$ satisfies the axioms of an abelian group with the zero element equal to $\frako_0$. In this way we equip the set  $\creg$ with an abelian group law. The points $V_0\cap\creg$ form a subgroup of $\creg$,  with cosets equal to $V_i\cap\creg$. The quotient group is isomorphic to the cyclic group $\Z/r\Z$. We have a group isomorphism
$$\creg \cong \Pic^0(C)\times \Z/r\Z \cong V_0\times \Z/r\Z.$$
Notice that the group $\creg$ acquires (non-canonically) a structure of a commutative algebraic group with connected component $V_0$  isomorphic to $\Pic^0(C)$. In fact $\Pic^0(C)$ has the structure of a commutative algebraic group (the \emph{generalized Jacobian} of $C$) for any (even non-reduced) projective algebraic curve \cite{Oort}.

If $C$ is a nonsingular curve, we immediately see that the group law coincides with the usual group law on an elliptic curve as defined, for example, in \cite{Hartshorne}. The group law on the component $V_0\cap\creg$ is the same as the group law obtained by the transfer of the group law on $\Pic^0(C)$ by means of the map $\kappa_0$ defined by the point $\frako_j$.

Let us describe  the group $\Aut(C)$ of automorphisms of $C$ in terms of the group law on each component $V_i\cap\creg$ isomorphic to $\Pic(V_0)$ considered as a one-dimensional algebraic group.  The group $\Aut(C)$ acts naturally on $\Pic(C)$ by  $\calL\to (\sigma^{-1})^*\calL,\  \sigma\in \Aut(C)$. In divisorial notation, $\sigma$ sends  $[D]$ to $ [\sigma(D)]$, where $\sum m_ix_i \mapsto \sum m_i\sigma(x)$. This action preserves the degree and the multi-degree. Thus it defines a homomorphism of groups
$$a:\Aut(C) \to \Aut(\Pic^0(C)).$$
The group $\Aut(\Pic^0(C))$, where $\Pic^0(C)$ is considered as a one-dimensional algebraic group is of course well-known. We have three different cases for $\Pic^0(C)$: an elliptic curve, or $\bbK^*$, or $\bbK$. Note that our automorphisms are automorphisms of the corresponding algebraic groups. In the first case, 
$$\Aut_{\gr}(\Pic^0(C)) \cong \begin{cases}\Z/2\Z&\text{if} \ j(C)\ne 0, 1728,\\
 \Z/4\Z&\text{if}\ j(C) = 1728, \cha(\bbK) \ne 2,3,\\
 \Z/6\Z&\text{if}\ j(C) = 0, \cha(\bbK) \ne 2,3,\\
  \Z/12\Z &\text{if}\ j(C) = 0 = 1728, \cha(\bbK) = 3,\\
 \Z/24\Z &\text{if}\ j(C) = 0 = 1728, \cha(\bbK) = 2.\end{cases}
 $$
(see \cite{Silverman}, Chap. III, \S 10). Here $j(C)$ is the absolute invariant of $C$ defined via the Weierstrass equation.  If $\bbK = \C$ then $\pico(C) \cong  (\C/\Gamma,+)$ for some discrete subgroup $\Gamma$, and a group automorphism of $\pico(C)$ is given by $z\mapsto \lambda z$, for some $\lambda\in \C^*$ such that $\lambda\Lambda = \Lambda$. 

We also have 
$$\Aut_{\textup{gr}}(\bbK^*) \cong \Z/2\Z, \quad \Aut_{\textup{gr}}(\bbK^+) \cong \bbK^*.$$

Let  $\sigma\in \Aut(C)$.  Then $\sigma(V_i) = V_{\tau(i)}$ for some permutation $\tau$ of $\{0,\dots,r-1\}$.  Our identifications $\kappa_i:V_i\cap\creg\to \Pic^0(C)$ induce maps 
$$
\kappa_{\tau(i)}\circ \sigma\circ \kappa_i^{-1}:\Pic^0(C)\to \Pic^0(C),\ 
[D] \mapsto a_\sigma([D])+\sigma(\frako_i)-\frako_{\tau(i)}.
$$
for each index $i$.  Each of these is an affine automorphism of $\Pic^0(C)$, given by composition of the above group automorphism $a_\sigma$ with translation by the divisor class 
$$
\frakb_i(\sigma) = [\sigma(\frako_i)-\frako_{\tau(i)}].
$$
We can therefore view the restriction of $\sigma$ to $V_i\cap\creg$ as  an ``affine automorphism'' $(a_\sigma, b_i(\sigma))$
\beq\label{affine}
\sigma(x) \sim a_\sigma([x_i-\frako_i])+\sigma(\frako_i) = a_\sigma([x_i-\frako_i])+ b_i(\sigma),
\eeq 
where $b_i(\sigma) := \kappa_{\tau(i)}^{-1}(\frakb_i(\sigma)) = \sigma(\frako_i)$.  It is clear that $\sigma$ is an affine automorphism of the whole group $\creg$ if and only if  $\frakb_i(\sigma)\in \Pic^0(C)$ is the same for all $i\in \Z / r\Z$. Or, in other words, $\sigma(\frako_i) -\frako_{\tau(i)}$ is a constant function from $\Z/r\Z$ to $\Pic^0(C)$.  Likewise, $\sigma$ defines a group automorphism of $C_{reg}$ if and only if the permutation $\tau:\Z/r\Z\to \Z/r\Z$ is a (group) automorphism and $\sigma(\frako_i) = \frako_{\tau(i)}$ for each $i$.

Finally let us discuss the special case of the group law on a reduced plane cubic curve, i.e. the case when $X = \cp^2$. By the adjunction formula, such a curve has arithmetic genus 1. So all of the above discussion applies with, of course, $r\leq 3$.  

\begin{prop}\label{P1} Assume that  $C$ is not isomorphic to the irreducible cuspidal cubic in characteristic 3 defined by the equation
$t_0t_2^2+t_1^3+t_1^2t_2 = 0$.   One can choose the points $\frako_i$ in such a way that 
 for any $x,y,z\in \creg$, no two lying on the same degree 1 component, 
\beq\label{p1}
x\oplus y\oplus z = 0 \Leftrightarrow x,y,z \ \text{are collinear}.
\eeq
\end{prop}

\begin{proof} Recall that an inflection point on a reduced plane algebraic curve is a nonsingular point such that there exists a line which intersects  the curve at this point with multiplicity $\ge 3$. Suppose $C$ is an irreducible cubic curve.  If $C$ is nonsingular, we can reduce the equation of $C$ to its Weierstrass form  (see \cite{Silverman}) and find the inflection point at infinity. If $C$ is an irreducible nodal curve, we can reduce the equation of $C$ to the form 
$t_0t_1t_2 +t_1^3-t_2^3 = 0$. If $\cha(\bbK) \ne 3$, we find 3 inflection points $(0,1,\epsilon)$, where $\epsilon^3 =1$. If $\cha(\bbK) = 3$ there is only one  inflection point $(0,1,1)$. 

If $C$ is a cuspidal  curve, we can reduce it to the form 
$t_0t_2^2+t_1^3 = 0$ if $\cha(\bbK) \ne 3$. If $\cha(\bbK) = 3$, there is one more  isomorphism class represented by the curve  $t_0t_2^2+t_1^3+t_1^2t_2 = 0$. The curve $t_0t_2^2+t_1^3 = 0$ has  the inflection point $(0,0,1)$. In the second case, there are no inflection points.  

Choose the points $\frako_i$ such that the divisor $\sum \frako_i\deg V_i$ is cut out by a line $\ell_0$. This means that $\frako_0$ is an inflection point if $C$ is irreducible, or the line $\ell_0$ is a tangent line to the point $\frako_i$ on the component $V_i$ of degree 1.  We also choose $V_0$ to be a line component if $C$ is reducible.

Assume for the moment that $C$ is irreducible.  Then $x\oplus y\oplus z = 0$ means that $x+y+z \sim 3\frako$.  From our choice of $\frako$, we infer $\calO_C(x+y+z)\cong \calO_C(3\frako) \cong \calO_C(1)$.  
Hence $x+y+z$ is cut out by a line.  Reversing the logic, concludes the proof for irreducible $C$.

Now assume that $C$ is reducible and that $x\in V_{i_x}$, etc.  Then $x\oplus y\oplus z = 0$ only if $i_x + i_y + i_z = 0$ in $\Z/r\Z$.  Therefore, since no two of the points lie in the same linear component, we cannot have $i_x = i_y = i_z = 0$.  The same is true if we assume instead that $x+y+z$ is a divisor cut out by a line.  

Similarly, if $i_x = i_y \neq i_z$ then $r=2$ with $\deg V_0=1$ and $\deg V_1 = 2$, and $i_z=0$, $i_x=i_y=1$.  So $x\oplus y\oplus z = 0$ becomes $x+y+z \sim 2\frako_0 + \frako_1$, and the argument concludes as in the irreducible case.  We leave case where $i_x,i_y,i_z$ are all different to the reader.  
\end{proof}

\begin{cor} 
\label{projauts}
An automorphism $\sigma$ of a plane cubic $C$ defined by affine automorphisms 
$(a_\sigma,b_i(\sigma))$, $i\in \Z/r\Z$, is a projective automorphism if and only if $\sum b_i(\sigma)\deg V_i$ is cut out by a line, or, in other words, 
$$\bigoplus_{i\in \Z/r\Z} b_i(\sigma)\deg V_i = 0$$
in the group law on $\creg$.
\end{cor}

\begin{proof} Let $\calO_C(1)$ be the restriction of $\calO_{\cp^2}(1)$ to $C$.  An automorphism $\sigma$ of $C$ is projective if and only it $\sigma^*(\calO_C(1)) \cong \calO_C(1)$. Since 
$\calO_C(1) \cong \calO_C(\sum \frako_i\deg V_i)$, this is equivalent to the condition that 
$\sum \sigma(\frako_i)\deg V_i$ is cut out by a line. But
$(a(\sigma,b_i(\sigma)(\frako_i) \sim  a_\sigma([\frako_i-\frako_i])+b_i(\sigma) = b_i(\sigma).$ This proves the assertion.
\end{proof}

\begin{rem}  Let us look at $\Aut_\textup{gr}(\creg)$ in more detail. We already know the structure of this group in the case when $C$ is irreducible. Assume that $C= V_0+V_1$ and $V_0$ intersects $V_1$ transversally. Then the tangent line $\la \frako,\frako_1\ra$ to $V_1$ is mapped under a group automorphism to the tangent line $\la \frako,\sigma(\frako_1)\ra$. If $\cha(\bbK)\ne 2,$ then there are two tangents line to a conic passing through a fixed point not on a conic. If $\cha(K) = 2,$ there is a unique point such that each line passing through this point is a tangent line. Since $V_0$ contains $\frako$ and is not tangent to $V_1$,  this case does not occur. If $\sigma(\frako_1) = \frako_1$, then $\sigma$ leaves four  lines invariant, the component $V_0$, two tangent lines, and the line joining the tangency points (the polar line of $\frako$ with respect to $V_1$). This implies that $\sigma$ is the identity and  easily shows that  
$$\Aut_\textup{gr}(\creg) \cong \Aut_\textup{gr}(V_0\cap\creg)\times \Z/2\Z \cong (\Z/2\Z)^2.$$ 
Next we assume that $V_0$ is tangent to $V_1$. One can reduce the equation of $C$ to the form $t_1(t_0^2-t_1t_2) = 0$ and assume that $\frako = (1,0,a)$, where $a = 0$  if $\cha(\bbK) \ne 2$ and $a = 0$ or $1$ otherwise. 
Easy computations find the group of automorphisms of the curve $C$. They consist of projective transformations $[x_0,x_1,x_2]\mapsto [\alpha x_0,x_1,\alpha^2 x_2]$ if $\cha(\bbK)\ne 2$ or $\cha(\bbK) = 2$ and $a= 1$. In the remaining case, the group consists of transformations 
$[x_0,x_1,x_2]\mapsto [\alpha x_0+\beta x_1,x_1,\beta^2x_1+\alpha^2 x_2]$. The natural homomorphism 
$\Aut_{\gr}(\creg) \to \Aut_{\gr}(V_0\cap\creg)$ 
is surjective and its kernel is trivial in the first case, and isomorphic to $\bbK^+\times \Z/2\Z$ in the second case. 

Finally, let us assume that $C$ is the union of three lines. We reduce the equation of $C$ to $t_0t_1t_2 = 0$ or $t_1t_2(t_1+t_2) = 0$ and compute the group of projective automorphisms leaving the point $\frako = (0,1,0)$ invariant. Easy computations show that $\Aut_{\gr}(\creg) \to \Aut_{\gr}(V_0\cap\creg)$ is surjective, and the kernel is isomorphic to $\Z/2\Z$ in the first case, and $\bbK^+\times \Z/2\Z$ in the second case.
\end{rem}

\bibliographystyle{/home/jeff/tex/mjo}
\bibliography{refs}
\end{document}